\RequirePackage{xcolor} 
\documentclass[a4paper]{easychair}

\usepackage[utf8]{inputenc}
\usepackage[english]{babel}
\usepackage{amsmath}
\usepackage{amsthm}
\usepackage{amsfonts}
\usepackage{amssymb}
\usepackage{latexsym}
\usepackage{relsize}
\usepackage{pgf,tikz}
\usepackage{tikz-cd}
\usetikzlibrary{arrows}
\usepackage{graphicx} 
\usepackage{hyperref}
\usepackage{doc}
\usepackage{enumitem} 

\usepackage{comment}

\newcommand*{\fregiu}{\downarrow \!}

\newcommand*{\fresu}{\uparrow \!}

\newtheorem{theorem}{Theorem}[section]
\newtheorem{lemma}[theorem]{Lemma}
\newtheorem{proposition}[theorem]{Proposition}

\newtheorem*{observation*}{Observation}

\theoremstyle{definition}
\newtheorem{definition}[theorem]{Definition}
\newtheorem{example}[theorem]{Example}

\newtheorem{remark}[theorem]{Remark}
\newtheorem*{remark*}{Remark}

\newcounter{claimcounter}
\newenvironment{claim}{\stepcounter{claimcounter}{\par\noindent\underline{Claim \theclaimcounter:}}}{}
\newenvironment{claimproof}[1]{\par\noindent\underline{Proof of Claim \theclaimcounter:}\space#1}{\hfill $\blacksquare$ \\}

\makeatletter
\newcommand{\subjclass}[2][2010]{%
  \let\@oldtitle\@title%
  \gdef\@title{\@oldtitle\footnotetext{#1 \emph{Mathematics subject classification.} #2.}}%
}
\newcommand{\keywords}[1]{%
  \let\@@oldtitle\@title%
  \gdef\@title{\@@oldtitle\footnotetext{\emph{Key words and phrases.} #1.}}%
}
\makeatother

\title{Existentially Closed Brouwerian Semilattices}
\author{Luca Carai         \and
        Silvio Ghilardi  
        }
        
\institute{
	Universit\`{a} degli Studi di Milano,
	Milano, Italy\\
	\email{luca.carai@studenti.unimi.it}\\
	\email{silvio.ghilardi@unimi.it}
}
\authorrunning{Carai and Ghilardi}
\titlerunning{Existentially Closed Brouwerian Semilattices}

\subjclass{Primary 03G25; Secondary 03C10, 06D20}
\keywords{Brouwerian semilattices, existentially closed structures, finite
duality}

\begin{document}

\definecolor{ffffff}{rgb}{1.,1.,1.}

\maketitle

\date{\today}
\begin{center}
\makeatletter
\@date
\makeatother
\end{center}

\begin{abstract}
 The variety of Brouwerian semilattices is amalgamable and locally finite, hence by well-known results~\cite{whe76}, it has a  model completion (whose models are the existentially closed structures). In this paper, we supply for such a model completion  a finite and rather simple axiomatization. 
\end{abstract}

\section{Introduction} \label{section:introduction}

In algebraic logic some attention has been paid to the class of existentially closed structures in varieties coming from the algebraization of common propositional logics. In fact, there are relevant cases where such classes are elementary: this includes, besides the easy case of Boolean algebras, also Heyting algebras~\cite{GZ_APAL,gz02}, diagonalizable algebras~\cite{Shavrukov,gz02} and some universal classes related to temporal logics~\cite{GvG},\cite{GvGlics}. However, very little is known about the related axiomatizations, with the remarkable exception of the 
case of the locally  finite amalgamable varieties of Heyting algebras recently investigated in~\cite{dj10} and of the simpler cases of posets and semilattices studied in~\cite{ab86}. In this paper, we use a 
methodology similar to~\cite{dj10} (relying on classifications of minimal extensions) in order to investigate the case of Brouwerian semilattices, i.e. the algebraic structures corresponding to the implication-conjunction fragment of intuitionistic logic. We 
obtain the finite axiomatization reported below, which is similar in spirit to the axiomatizations from~\cite{dj10} (in the sense that we also have kinds of `density' and 
`splitting' conditions). The main technical problem  we must face for this result (making axioms formulation slightly more complex and proofs much more involved) 
is the lack of joins in the language of Brouwerian semilattices. 

\vspace{-6.5pt} 

\subsection{Statement of the main result}
The first researcher to consider the Brouwerian semilattices as algebraic objects in their own right was W. C. Nemitz in \cite{nem65}.
A \emph{Brouwerian semilattice} is a poset $(P, \leq)$ having a greatest element (which we denote with $1$),
inf's of pairs (the inf of $\lbrace a,b \rbrace$ is called `meet' of $a$ and $b$ and denoted with $a \wedge b$)
and relative pseudo-complements (the relative pseudo-complement of $a$ and $b$ is denoted with $a \rightarrow b$).
We recall that $a\rightarrow b$ is characterized by the the following property: for every $c \in P$ we have
\begin{equation*}
c \leq a \rightarrow b \quad \text{iff} \quad c \wedge a \leq b.
\end{equation*}

Brouwerian semilattices can also be defined in an alternative way as algebras over the signature $1, \wedge, \rightarrow$, subject to the following equations
\begin{equation*}
\begin{array}{r c l}
a \wedge a=a & \qquad & a \wedge (a \rightarrow b)= a \wedge b \\
a \wedge b = b \wedge a & & b \wedge (a \rightarrow b) = b \\
a \wedge (b \wedge c) = (a \wedge b) \wedge c & & a \rightarrow (b \wedge c) = (a \rightarrow b) \wedge (a \rightarrow c) \\
a \wedge 1= a & & a \rightarrow a=1
\end{array}
\end{equation*}
In case this equational axiomatization is adopted, the partial order $\leq$ is recovered via the definition
$ a \leq b$ iff $a \wedge b= a$.

By a result due to Diego and McKay~\cite{die66,mck68}, Brouwerian semilattices are locally finite (meaning that all finitely generated Brouwerian semilattices are finite); since they are also amalgamable, it follows~\cite{whe76} that the theory of Brouwerian semilattices has a model completion.
We prove that such a model completion 
is given by the above set of axioms for the theory of Brouwerian semilattices together with the three additional axioms (Density1, Density2, Splitting) below.\\

\noindent We use the shorthand $a \ll b$ to mean that $a \leq b$ and $b \rightarrow a=a$.\\

\noindent \textbf{[Density 1]} For every $c$ there exists an element $b$ different from $1$ such that $b \ll c$. \\

\noindent \textbf{[Density 2]} For every $c,a_1,a_2,d$ such that $a_1,a_2 \neq 1$, $a_1 \ll c$, $a_2 \ll c$ and $d \rightarrow a_1=a_1$, $d \rightarrow a_2=a_2$ there exists an element $b$ different from $1$ such that:
\begin{equation*}
\begin{split}
	a_1 \ll b \\
	a_2 \ll b \\
	b \ll c \\
    d \rightarrow b=b
\end{split}
\end{equation*}

\noindent\textbf{[Splitting]} For every $a,b_1,b_2$ such that $1 \neq a \ll b_1 \land b_2$ there exist elements $a_1$ and $a_2$ different from $1$ such that:
\begin{equation*}
\begin{split}
	 & b_1 \geq a_1, \: b_2 \geq a_2 \\
     & a_2 \rightarrow a = a_1  \\
     & a_1 \rightarrow a = a_2 \\
     & a_2 \rightarrow b_1 = b_2 \rightarrow b_1 \\
     & a_1 \rightarrow b_2 = b_1 \rightarrow b_2 
\end{split}
\end{equation*}


As testimony of the usefulness of this result, the following proposition shows some properties of the existentially closed Brouwerian semilattices that can be deduced from our investigation as an easy exercise.

\begin{proposition}
Let $L$ be an existentially closed Brouwerian semilattice. Then:
\begin{enumerate}
\item $L$ has no bottom element.
\item If $a,b \in L$ are incomparable, i.e. $a \nleq b$ and $b \nleq a$, then the join of $a$ and $b$ in $L$ does not exist.
\item There are no meet-irreducible elements in $L$.
\end{enumerate}
\end{proposition}

The paper is structured as follows: Section \ref{section:preliminary} gives the basic notions and definitions, in particular it describes the finite duality and characterizes the existentially closed structures by means of embeddings of finite extensions of finite sub-structures. In Section \ref{section:minimalfiniteextensions} we investigate the minimal extensions and use them to give an intermediate characterization of the existentially closed structures. Section \ref{section:axioms} focuses on the axiomatization, it is split into two subsections: the first about the Splitting axiom and the second about the Density axioms. Finally, in Section \ref{section:properties} we present and prove some properties of the existentially closed structures whose validity follows from this investigation.

\section{Preliminary Background} \label{section:preliminary}

A \textit{co-Brouwerian semilattice}, CBS for short, is a structure obtained by reversing the order of a Brouwerian semilattice.\\
We will work with CBSes instead of Brouwerian semilattices.

\begin{definition}
A poset $(P, \leq)$ is said to be a \textit{co-Brouwerian semilattice} if it has a least element, which we denote with $0$, and for every $a,b \in P$ there exists the sup of $\lbrace a,b \rbrace$, which we call join of $a$ and $b$ and denote with $a \vee b$, and the difference $a - b$ satisfying for every $c \in P$
\[
a-b \leq c \quad \text{ iff } \quad a \leq b \vee c.
\]
$a \ll b$ will mean that $a \leq b$ and $b-a=b$.
\end{definition} 

Clearly, there is also an alternative equational definition for co-Brouwerian semilattices (which we leave to the reader, because it is dual to the equational definition for Brouwerian semilattices given above).\\

Moreover, we will call co-Heyting algebras the structures obtained reversing the order of Heyting algebras. Obviously any co-Heyting algebra is a CBS.

\begin{definition}
Let $A,B$ be co-Brouwerian semilattices.
A map $f:A \to B$ is a \textit{morphism of co-Brouwerian semilattices} if it preserves $0$, the join and difference of any two elements of $A$.
\end{definition}

Notice that such a morphism  $f$ is an order preserving map because, for any $a,b$ elements of a co-Brouwerian semilattice, we have $a \leq b$ iff $a \vee b= b$.

\begin{definition}
Let $L$ be a CBS.\\
We say that $g \in L$ is \textit{join-irreducible} iff for every $n \geq 0$ and $b_1, \ldots, b_n \in L$, we have that
\[
g \leq b_1 \vee \ldots \vee b_n  \quad \text{implies} \quad g \leq b_i \text{ for some } i=1, \ldots, n.
\]
Notice that taking $n=0$ we obtain that join-irreducibles are different from $0$.
\end{definition}

\begin{remark}
Let $L$ be a CBS and $g \in L$. Then the following conditions are equivalent:
\begin{enumerate}
\item $g$ is join-irreducible. \label{itm:propjoinirr1}
\item $g \neq 0$ and for any $b_1,b_2 \in L$ we have that $g \leq b_1 \vee b_2 \;$ implies $\; g \leq b_1$ or $g \leq b_2$. \label{itm:propjoinirr2}
\item For every $n \geq 0$ and $b_1, \ldots, b_n \in L$ we have that\\
 $g = b_1 \vee \ldots \vee b_n \;$ implies $\; g = b_i$ for some $i=1, \ldots, n$. \label{itm:propjoinirr3}
\item $g \neq 0$ and for any $b_1,b_2 \in L$ we have that $g = b_1 \vee b_2 \;$ implies $\; g = b_1$ or $g = b_2$. \label{itm:propjoinirr4}
\item $g \neq 0$ and for any $a \in L$ we have that $g-a=0 \;$ or $\; g-a=g$. \label{itm:propjoinirr5}
\end{enumerate}
\end{remark}

\begin{proof}
The implications \ref{itm:propjoinirr1} $\Leftrightarrow$ \ref{itm:propjoinirr2}, \ref{itm:propjoinirr3} $\Leftrightarrow$ \ref{itm:propjoinirr4} and \ref{itm:propjoinirr1} $\Rightarrow$ \ref{itm:propjoinirr3} are straightforward. For the remaining ones see Lemma 2.1 in~\cite{koh81}.
\end{proof}

\begin{definition}
Let $L$ be a CBS and $a \in L$.\\
A \textit{join-irreducible component} of $a$ is a maximal element among the join-irreducibles of $L$ that are smaller than or equal to $a$.
\end{definition}

\begin{remark} \label{listsimplefacts}
The following is a list of facts that might be used without explicit mention.\\
These identities hold in any CBS:
\begin{equation*}
\begin{split}
0-a=0 \qquad & a-0=a \\
(a-b) \vee b=a \vee b \qquad & (a-b) \vee a= a \\
(a-b) \vee (a-(a-b))=a \qquad & a-(a-(a-b))=a-b \\
(a_1 \vee \cdots \vee a_n)-b= & (a_1-b)\vee \cdots \vee (a_n-b) \\  
a-(b_1 \vee \cdots \vee b_n)= & ((a-b_1)- \cdots )-b_n
\end{split}
\end{equation*}
In particular
\[
(a-b)-c=(a-c)-b
\]
Furthermore in any CBS:
\begin{align*}
a \leq b \qquad & \text{iff} \qquad a-b=0 \\
\text{if } \; b \leq c \; \text{ then } \; b-a & \leq c-a \; \text{ and } \; a-c \leq a-b 
\end{align*}
The following facts are true in any finite CBS:
\begin{equation*}
\begin{split}
& a = \bigvee \lbrace \text{join-irreducible components of } a  \rbrace \\
a-b=  \bigvee \lbrace g \: | \: g & \text{ is a join-irreducible component of } a \text{ such that } g \nleq b  \rbrace 
\end{split}
\end{equation*}
Moreover, in a finite CBS, $g$ is join-irreducible iff it has a unique predecessor, i.e. a maximal element among the elements strictly smaller than $g$, and in that case we denote it by $g^-$ and it is equal to $\bigvee_{a < g} a$.\\
Recall that $a \ll b$ means $a \leq b$ and $b-a=b$. Thus, in any finite CBS, $a \ll b$ if and only if $a \leq b$ and there are no join-irreducible components of $b$ that are less than or equal to $a$. Finally, if $g$ is join-irreducible then $ g^- \ll g$.
\end{remark}

\begin{example}
Let $(P, \leq)$ be a poset. For any $a \in P$ we define $\fregiu a = \lbrace p \in P \: | \: p \leq a \rbrace$ and for any $A \subseteq P$ we define $\fregiu A= \bigcup_{a \in A} \fregiu a$. A subset $D \subseteq P$ such that $D= \fregiu D$ is called a \textit{downset}, i.e. a downward closed subset, of $P$. The downsets $\fregiu a$ and $\fregiu A$ are called the \textit{downsets generated by} $a$ and $A$.\\
Given a poset $P$, the set of downsets of $P$ denoted by $\mathcal{D}(P)$ has naturally a structure of CBS given by the usual inclusion of subsets. Joins coincide with the union of subsets and the zero element with the empty subset. It turns out that the difference of two downsets $A,B \in \mathcal{D}(P)$ is $A-B= \fregiu (A \setminus B)$.\\
Note that if $P$ is finite then also $\mathcal{D}(P)$ is. In that case any downset $A \in \mathcal{D}(P)$ is generated by the set of its maximal elements and for any $A,B \in \mathcal{D}(P)$ we have that $A-B$ is the downset generated by the maximal elements of $A$ that are not in $B$. Moreover the join-irreducibles of $\mathcal{D}(P)$ are the downsets of the form $\fregiu p$ for $p \in P$ and the downsets generated by the maximal elements of a given downset are its join-irreducible components. Notice that this is not always the case when $P$ is infinite.\\
Finally, when $P$ is finite, for $A,B \in \mathcal{D}(P)$ satisfying $A \ll B$ means that $A \subseteq B$ and $A$ does not contain any maximal element of $B$. 
\end{example}

\subsection{Locally finiteness}

\begin{theorem} \label{faclocfin}
The variety of CBSes is locally finite.
\end{theorem}

\begin{proof} We just sketch the proof 
first presented in \cite{mck68}.
A CBS $L$ is subdirectly irreducible iff $L \setminus \lbrace 0 \rbrace$ has a least element, or equivalently $L$ has a single atom, i.e. a minimal element different from $0$.\\
Let $L$ be subdirectly irreducible and $u$ the least element of $L \setminus \lbrace 0 \rbrace$. Then $L \setminus \lbrace u \rbrace$ is a sub-CBS of $L$. This implies that any generating set of $L$ must contain $u$.\\
Moreover if $L$ is generated by $n$ elements then $L \setminus \lbrace u \rbrace$ can be generated by $n-1$ elements. It follows that the cardinality of subdirectly irreducible CBSes generated by $n$ elements is bounded by $\# F_{n-1} +1$ where $F_m$ is the free CBS on $m$ generators. Since $\# F_0=1$ by induction we obtain that $F_m$ is finite for any $m$ because it is a subdirect product of a finite family of subdirectly irreducibles which are generated by $m$ elements.
\end{proof}

Computing the cardinality of $F_m$ is a hard task. It is known that $\# F_0=1, \# F_1=2, \# F_2=18$ and $\# F_3=623,662,965,552,330$. The size of $F_4$ is still unknown. In \cite{koh81} it is proved that the number of join-irreducible elements of $F_4$ is $2,494,651,862,209,437$.
This shows that although the cardinality of the free CBS on a finite number of generators is always finite,  it grows very rapidly.

\subsection{Finite duality}

\begin{proposition}
Any finite CBS is a distributive lattice. 
\end{proposition}

\begin{proof} A finite CBS is complete, hence also co-complete, so it is a lattice.
The map $a \vee (-)$ preserves infima because it has a left adjoint given by $ (-) - a$. Thus the distributive laws hold.
\end{proof}

\begin{remark}
Every finite Brouwerian semilattice is a Heyting algebra but it is not true that every Brouwerian semilattices morphism among finite Brouwerian semilattices is a Heyting algebra morphism. 
\end{remark}

The following theorem presents the finite duality result due to K\"{o}hler: 

\begin{theorem} \label{th:dualitybs}
The category $ \mathbf{CBS}_{fin}$ of finite CBSes is dual to the category $\mathbf{P}$ whose objects are finite posets and whose morphisms are partial mappings $\alpha:P \rightarrow Q$ satisfying:
\begin{enumerate}[label=(\roman*)] 
\item $\forall p,q \in \text{dom }  \alpha \text{ if } p < q \text{ then } \alpha(p) < \alpha(q)$. \label{itm:defpmorp1}
\item $\forall p \in \text{dom }  \alpha \text{ and } \forall q \in Q \text{ if } \alpha(p) < q \text{ then } \exists r \in \text{dom } \alpha \text{ such that } p < r \text{ and } \alpha(r)=q$. \label{itm:defpmorp2}
\end{enumerate}
\end{theorem}
\begin{proof}
 The proof can be found in \cite{koh81}. We just recall how the equivalence works.
To a finite poset $P$ it is associated the CBS $\mathcal{D}(P)$ of downsets of $P$.\\
To a $\mathbf{P}$-morphism among finite posets it is associated the morphism of CBSes that maps a downset to the downset generated by its preimage. More explicitly, to a $\mathbf{P}$-morphism $f:P \to Q$ is associated the morphism that maps a downset $D$ of $Q$ to $\fregiu f^{-1}(D)=\{ p\in P\mid \exists p'\geq p\; (p'\in \text{dom} f ~\&~f(p')\in D)\}$.\\
On the other hand, to a finite CBS $L$ it is associated the poset of its join-irreducible elements.
\end{proof}

The following proposition is easily checked:

\begin{proposition}
Let $P,Q$ be finite posets and $f : P \rightarrow Q$ a $\mathbf{P}$-morphism. Let $\alpha$ be the associated morphism of CBSes. Then
\begin{enumerate}[label=(\roman*)] 
\item $\alpha$ is injective if and only if $f$ is surjective.
\item $\alpha$ is surjective if and only if $\text{dom }  f = P$ and $f$ is injective.
\end{enumerate}
\end{proposition}

Duality results involving all Brouwerian semilattices can be found in the recent paper~\cite{bez13} due to G. Bezhanishvili and R. Jansana. Other dualities are described in~\cite{vra86} and~\cite{cel03}.\\

Using finite duality we can show that the variety of CBSes has the amalgamation property.\\
The amalgamation property for Brouwerian semilattices is the algebraic counterpart of a syntactic fact about the implication-conjunction fragment of intuitionistic propositional logic: the interpolation property. The proof that such fragment satisfies this property can be found in~\cite{ren89}.

\begin{theorem} 
The theory of CBSes has the amalgamation property.
\end{theorem}

\begin{proof}
First, we show that the pushout of given monomorphisms (= injective maps) $m:L_0 \to L_1$ and $n:L_0 \to L_2$ among \textit{finite} CBSes is still formed by monomorphisms. Then we extend the result to the general case.\\
To do this, by finite duality, it is sufficient to show that the category $\mathbf{P}$ has the coamalgamation property.
This means that, given two surjective $\mathbf{P}$-morphisms among finite posets $f:P \to Q$ and $g:R \to Q$ there exist a finite poset $S$ and two surjective $\mathbf{P}$-morphisms $f':S \to R$ and $g':S \to P$ making the following diagram commute.
\[
\begin{tikzcd}
S \arrow[r, "f'"] \arrow[d, "g'"'] & R \arrow[d, "g"] \\
P \arrow[r, "f"'] & Q
\end{tikzcd}
\]
For any $p \in P$ let $q_1, \ldots, q_n$ be the minimal elements of
\[
\lbrace f(a) \: | \: a \in \text{dom } f \text{ and } a \geq p \rbrace \subseteq Q
\]
it could be that $n=0$ when such set is empty. Define:
\[
S_p= \lbrace (\lbrace p \rbrace, \lbrace r_1, \ldots, r_n \rbrace) \: | \: r_i \in \text{dom } g \text{ and } g(r_i)=q_i \text{ for } i=1, \ldots,n \rbrace
\]
Analogously for any $r \in R$ let $q_1, \ldots, q_n$ be the minimal elements of
\[
\lbrace f(a) \: | \: a \in \text{dom } g \text{ and } a \geq r \rbrace  \subseteq Q
\]
and define
\[
S_r= \lbrace (\lbrace p_1, \ldots, p_n \rbrace, \lbrace r \rbrace) \: | \: p_i \in \text{dom } f \text{ and } f(p_i)=q_i \text{ for } i=1, \ldots,n \rbrace
\]
Let
\[
S_P= \bigcup_{p \in P} S_p \qquad S_R= \bigcup_{r \in R} S_r
\]
And take $S= S_P \cup S_R$.\\
We can immediately observe that if $p \in \text{dom } f$ then 
\[
S_p= \lbrace (\lbrace p \rbrace, \lbrace r \rbrace) \: | \: r \in \text{dom } g \text{ and } f(p)=g(r) \rbrace.
\]
And that
\[
S_P \cap S_R= \lbrace (\lbrace p \rbrace, \lbrace r \rbrace) \: | \: p \in \text{dom } f,\, r \in \text{dom } g  \text{ and } f(p)=g(r) \rbrace.
\]
And finally that if $(\lbrace p \rbrace, \lbrace r_1, \ldots, r_n \rbrace) \in S_p$ then the $r_i$'s are two-by-two incomparable, indeed $g$ is order preserving and the $g(r_i)$'s are incomparable since they are the minimal elements of a subset of $Q$. Thus the elements of the two components of any element of $S$ are two-by-two incomparable.\\
We define an order on $S$ in the following way:\\
let $(A_1,A_2), (B_1,B_2) \in S$ where $A_1,B_1 \subseteq P$ and $A_2,B_2 \subseteq R$, we define
\begin{equation*}
\begin{split}
(A_1,A_2) \leq (B_1,B_2) \qquad  \text{iff} \qquad & \forall \, y \in B_1 \: \exists \, x \in A_1 \text{ such that } x \leq y \\
& \text{and} \\
& \forall \, y \in B_2 \: \exists \, x \in A_2 \text{ such that } x \leq y
\end{split}
\end{equation*}
This order relation is clearly reflexive.\\
It is antisymmetric, indeed let $(A_1,A_2) \leq (B_1,B_2)$ and $(B_1,B_2) \leq (A_1,A_2)$ then for any $y \in B_1$ there exists $x \in A_1$ such that $x \leq y$ and there exists $z \in B_1$ such that $z \leq x$. Since the elements of $B_1$ are incomparable we get $z=y$ and thus $x=y$. Therefore $B_1 \subseteq A_1$. Symmetrically we get $A_1 \subseteq B_1$ and then $A_1=B_1$. Reasoning similarly we get $A_2=B_2$ and then $(A_1,A_2) = (B_1,B_2)$.\\
It is transitive, indeed let $(A_1,A_2) \leq (B_1,B_2)$ and $(B_1,B_2) \leq (C_1,C_2)$ then for any $z \in C_1$ there exists $y \in B_1$ such that $y \leq z$ and there exists $x \in A_1$ such that $x \leq y$ and hence also $x \leq z$. Analogously for the second components. Therefore $(A_1,A_2) \leq (C_1,C_2)$.\\
Thus we have defined a partial order on $S$.\\
Take $g':S \to P$ and $f':S \to R$ as:
\begin{equation*}
\begin{split}
\text{dom } g' = S_P \qquad \qquad & \text{dom } f' = S_R \\
g'(\lbrace p \rbrace, A_2)=p \qquad \qquad & f'(A_1, \lbrace r \rbrace)=r
\end{split}
\end{equation*}
Then
\begin{equation*}
\begin{split}
\text{dom } f \circ g' & = (g')^{-1} (\text{dom } f)  = \lbrace (\lbrace p \rbrace, A_2)\in S_P \: | \: p \in \text{dom } f \rbrace \\
& =\lbrace (\lbrace p \rbrace, \lbrace r \rbrace) \: | \: p \in \text{dom } f, \, r \in \text{dom } g \text{ and } f(p)=g(r) \rbrace  \\
& = (f')^{-1} (\text{dom } g)  = \lbrace (A_1, \lbrace r \rbrace)\in S_R \: | \: r \in \text{dom } g \rbrace \\
& = \text{dom } g \circ f'
\end{split}
\end{equation*}
and if $p \in \text{dom } f, \, r \in \text{dom } g \text{ and } f(p)=g(r)$ then
\[
(f \circ g')((\lbrace p \rbrace, \lbrace r \rbrace))=f(p)=g(r)=(g \circ f')((\lbrace p \rbrace, \lbrace r \rbrace))
\]
$g'$ is surjective: indeed let $p \in P$ and $q_1, \ldots, q_n$ be the minimal elements of $\lbrace f(a) \: | \: a \in \text{dom } f \text{ and } a \geq p \rbrace$, by surjectivity of $g$ there exist $r_1, \ldots, r_n \in \text{dom } g$ such that $g(r_i)=q_i$, then $(\lbrace p \rbrace, \lbrace r_1, \ldots, r_n \rbrace) \in S_p \subseteq \text{dom } g'$ and $g'((\lbrace p \rbrace, \lbrace r_1, \ldots, r_n \rbrace))=p$. Analogously for the surjectivity of $f'$.\\
It remains to show that $g',f'$ are $\mathbf{P}$-morphisms.\\
Let $(\lbrace p \rbrace, A), (\lbrace p' \rbrace, B) \in S_P= \text{dom } g'$ such that $(\lbrace p \rbrace, A) < (\lbrace p' \rbrace, B)$, we show that $p < p'$.\\
Clearly $p \leq p'$ by the definition of the order on $S$.\\
Suppose that $p=p'$, let $q_1, \ldots, q_n$ be the minimal elements of $\lbrace f(a) \: | \: a \in \text{dom } f \text{ and } a \geq p \rbrace$. Let $A=\lbrace r_1, \ldots, r_n \rbrace$ and $B=\lbrace r_1', \ldots, r_n' \rbrace$ be such that $g(r_i)=g(r_i')=q_i$ for $i=1, \ldots,n$. Then for any $r_i' \in B$ there exists $r_j \in A$ such that $r_j \leq r_i'$. If $r_j \leq r_i'$ with $j \neq i$ then $q_j=g(r_j) \leq g(r_i')=q_i$ and this is absurd because the $q_i$'s are incomparable. Therefore $r_i \leq r_i'$ for any $i=1, \ldots, n$, if $r_i < r_i'$ then $q_i=g(r_i) < g(r_i')=q_i$ which is absurd. Thus $r_i=r_i'$ and $A=B$, we have obtained a contradiction. Analogous for $f'$.\\
Therefore $g',f'$ preserve the strict order.\\
Let $(\lbrace p \rbrace, A) \in S_P$ and $p < p'$.\\
Let $q_1, \ldots, q_n$ be the minimal elements of $\lbrace f(a) \: | \: a \in \text{dom } f \text{ and } a \geq p \rbrace$ and $q_1', \ldots, q_m'$ be the minimal elements of $\lbrace f(a) \: | \: a \in \text{dom } f \text{ and } a \geq p' \rbrace$; since the latter set is included in the former and they are both finite we have that for any $q_j'$ there exist $q_i$ such that $q_i \leq q_j'$.\\
Let $A= \lbrace r_1, \ldots, r_n \rbrace$ with $g(r_i)=q_i$. Since $g$ is a $\mathbf{P}$-morphism and for any $q_j'$ there exists $i$ such that $g(r_i)=q_i \leq q_j'$, there exists $r_j' \in \text{dom } g$ such that $r_i \leq r_j'$ and $g(r_j')=q_j'$. Take $B= \lbrace r_1', \ldots, r_m' \rbrace$ then for any $r_j'$ there exists $r_i$ such that $r_i \leq r_j'$, therefore $(\lbrace p \rbrace, A) <  (\lbrace p' \rbrace, B) \in S_P$. Analogous for $f'$.\\
Thus $f',g'$ are surjective $\mathbf{P}$-morphisms and they coamalgamate $f,g$.\\
We now want to prove \textit{the general case}: pushouts of monos along monos in the category of CBSes are monos.\\
Suppose $m:L_0 \to L_1$ and $n:L_0 \to L_2$ are monos and $L_0,L_1,L_2$ are CBSes. Since the variety is locally finite by Theorem \ref{faclocfin}, we can consider $L_0,L_1,L_2$ as filtered colimits of families of finite CBSes. Assume without loss of generality that $L_1 \cap L_2 =L_0$ and $m,n$ are inclusions, then we can consider the families indexed by $\mathcal{P}_{fin}(L_1 \cup L_2)$ given for any finite subset $S \subseteq L_1 \cup L_2$ by the sub-CBSes respectively of $L_1,L_2$ and $L_0$ generated respectively by $S \cap L_1$, $S \cap L_2$ and $m^{-1}(S \cap L_1) \cap n^{-1}(S \cap L_2)=S \cap L_0$. Then we can compute the pushouts of the restrictions of the monos for any index, the colimit of all these pushouts is a mono because each of them is a mono. Thus we have obtained that the pushout of $m$ along $n$ and the pushout of $n$ along $m$  are monomorphisms.
\end{proof}

\subsection{Existentially closed CBSes} \label{subsectcharexclosed}

In this subsection we want to characterize the existentially closed CBSes using the finite extensions of their finite sub-CBSes.

\begin{definition}
Let $T$ be a first order theory and $\mathcal{A}$ a model of $T$.
$\mathcal{A}$ is said to be \emph{existentially closed for $T$} if for every model $\mathcal{B}$ of $T$ such that $\mathcal{A} \subseteq \mathcal{B}$ every existential sentence in the language extended with names of elements of $\mathcal{A}$ which holds in $\mathcal{B}$ also holds in $\mathcal{A}$
\end{definition}

The following proposition is well-known from textbooks~\cite{ck90}:

\begin{proposition}
Let $T$ be a universal theory. If $T$ has a model completion $T^*$, then the class of models of $T^*$ is the class of models of $T$ which are existentially closed for $T$.
\end{proposition}

Thanks to the locally finiteness and the amalgamability, by an easy model-theoretic reasoning we obtain the following characterization of the existentially closed CBSes:

\begin{theorem} \label{theoexclosedcharac}
Let $L$ be a CBS. $L$ is existentially closed iff for any finite sub-CBS $L_0 \subseteq L$ and for any finite extension $C \supseteq L_0$ there exists an embedding $C \to L$ fixing $L_0$ pointwise. 
\end{theorem}

\begin{proof}
First, we prove that if for any finite sub-CBS $L_0 \subseteq L$ and for any finite extension $C \supseteq L_0$ there exists an embedding $C \to L$ fixing $L_0$ pointwise, then $L$ is existentially closed.
Let $D$ be an extension of $L$ and $\exists x_1, \ldots, x_m \varphi (x_1, \ldots, x_m, a_1, \ldots,a_n)$ an existential $\mathcal{L}_L$-sentence, where $\varphi (x_1, \ldots, x_m, a_1, \ldots,a_n)$ is quantifier free and $a_1, \ldots, a_n \in L$.\\
Suppose $D \vDash \exists x_1, \ldots, x_m \varphi (x_1, \ldots, x_m, a_1, \ldots,a_n)$.\\
Let $d_1, \ldots, d_m$ be elements of $D$ such that $D \vDash \varphi (d_1, \ldots, d_m, a_1, \ldots,a_n)$.\\
Consider the sub-CBS $L_0 \subseteq L$ generated by $a_1, \ldots, a_n$ and the sub-CBS $C \subseteq D$ generated by $d_1, \ldots, d_m, a_1, \ldots,a_n$. They are both finite because they are finitely generated and the CBSes form a locally finite variety.\\
By hypothesis there exists an embedding $C \to L$ fixing $L_0$ pointwise.\\
Let $d_1', \ldots, d_m'$ be the images of $d_1, \ldots, d_m$ by this embedding. Thus $L \vDash \varphi (d_1', \ldots, d_m', a_1, \ldots,a_n)$ because $\varphi$ is quantifier free.\\
Therefore $L \vDash \exists x_1, \ldots, x_m \varphi (x_1, \ldots, x_m, a_1, \ldots,a_n)$. It follows that $L$ is existentially closed.\\
To prove the other implication, suppose $L$ is existentially closed.\\
By amalgamation property there exists a CBS $D$ amalgamating $L$ and $C$ over $L_0$.\\
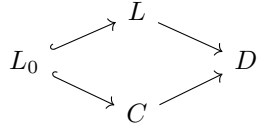
\[
\begin{tikzcd}[row sep=tiny]
& L \arrow[rd]  & \\
L_0 \arrow[rd, hook] \arrow[ru, hook] & & D  \\
& C \arrow[ru] & 
\end{tikzcd}
\]
Let $\Sigma$ be the set of quantifier free $\mathcal{L}_C$-sentences of the form $c * c'=c''$ true in $C$ where $c,c',c'' \in C$ and $*$ is either $\vee$ or $-$. Hence $(C,\Sigma)$ is a finite presentation of $C$.\\
Now let $c_1,\ldots,c_r,a_1, \ldots, a_n$ be an enumeration of the elements in $C$ where the $a_i$'s are the elements in $L$. We obtain the quantifier free $\mathcal{L}_C$-sentence $\sigma(c_1,\ldots,c_r,a_1, \ldots,a_n)$ by taking the conjunction of all the sentences in $\Sigma$ and all the sentences of the form $\neg (c = c')$ for every $c,c' \in C$ such that $c \neq c'$.\\
Clearly $\exists x_1, \ldots, x_r \sigma(x_1,\ldots,x_r,a_1, \ldots,a_n)$ is an existential $\mathcal{L}_L$-sentence true in $D$. Since $L$ is existentially closed, $L \vDash \exists x_1, \ldots, x_r \sigma(x_1,\ldots,x_r,a_1, \ldots,a_n)$. Let $c_1', \ldots, c_r' \in L$ be such that $L \vDash \sigma(c_1',\ldots,c_r',a_1, \ldots,a_n)$. The map $C \to L$ fixing $L_0$ pointwise and mapping $c_i$ to $c_i'$ is an embedding. Indeed it is injective and an homomorphism by definition of the sentence $\sigma$.
\end{proof}

\section{Minimal finite extensions} \label{section:minimalfiniteextensions}

In this section we focus on the finite extensions of CBSes. We are interested in particular to the minimal ones since any finite extension can be decomposed in a finite chain of minimal extensions. We will study minimal finite extensions by describing the properties of some elements which generate them. This investigation will lead us to another characterization of the existentially closed CBSes.

\begin{definition} \label{deforder}
Let $P$ be a poset, $P_0 \subseteq P$ and $\mathcal{F}$ a partition of $P_0$, let $A,B \in \mathcal{F}$.\\
We say that $A \leq B$ iff there exist $a \in A, b \in B$ such that $a \leq b$.
\end{definition}

\begin{proposition} \label{proppartition}
Let $P$ be a finite poset.\\
To give a surjective $\mathbf{P}$-morphism $f$ from $P$ to any finite poset is equivalent, up to isomorphism, to give a partition $\mathcal{F}$ of a subset of $P$ such that:
\begin{enumerate}
\item for all $A,B \in \mathcal{F}$ we have that if $A \leq B$ and $B \leq A$ then $A=B$, \label{itm:condpartantisym}
\item for all $A,B \in \mathcal{F}$ and $a \in A$ if $A \leq B$ then there exists $b \in B$ such that $a \leq b$,\label{itm:condpartstar}
\item for all $A \in \mathcal{F}$ we have that all the elements of $P$ in $A$ are two-by-two incomparable.\label{itm:condpartincomp}
\end{enumerate}
\end{proposition}

\begin{proof}
Given a surjective $\mathbf{P}$-morphism $f : P \rightarrow Q$, the partition $\mathcal{F}$ of $\text{dom } f \subseteq P$ is obtained by taking the collection of the fibers of $f$. 
$\mathcal{F}$ satisfies \ref{itm:condpartantisym} because $f$ is order preserving and the order on $Q$ is antisymmetric. Furthermore $\mathcal{F}$ satisfies \ref{itm:condpartstar} as a consequence of condition \ref{itm:defpmorp2} in the definition of $\mathbf{P}$-morphism. Finally, $\mathcal{F}$ satisfies \ref{itm:condpartincomp} because $\mathbf{P}$-morphisms are strict order preserving.\\
On the other hand, given a partition $\mathcal{F}$ of a subset $P_0$ of $P$ satisfying the conditions \ref{itm:condpartantisym}, \ref{itm:condpartstar} and \ref{itm:condpartincomp}, we obtain a poset $Q$ by taking the quotient set of $P_0$ given by $\mathcal{F}$ with the order defined in Definition \ref{deforder}. The partial map $f : P \rightarrow Q$ is just the projection onto the quotient.\\
$Q$ is a poset: the order of $Q$ is clearly reflexive, it is antisymmetric because $\mathcal{F}$ satisfies \ref{itm:condpartantisym}. It is also transitive because if $A \leq B$ e $B \leq C$ then there exist $a \in A, \ b, b' \in B, \ c \in C$ such that $a \leq b, \ b' \leq c$; since \ref{itm:condpartstar} holds, there exist $c' \in C$ such that $b \leq c'$, hence $a \leq c'$ and $A \leq C$. The projection $f$ is order preserving, it is a $\mathbf{P}$-morphism because \ref{itm:condpartstar} holds and it is obviously surjective.\\
It remains to show that a surjective $\mathbf{P}$-morphism $f:P \rightarrow Q$ differs by an isomorphism to the projection onto the quotient defined by the partition given by the fibers of $f$. This follows from the fact that for any $a,b \in \text{dom } f$ it is $f(a) \leq f(b)$ iff $f^{-1}(f(a)) \leq f^{-1}(f(b))$ (notice that $f^{-1}(f(a))$ is the element of $\mathcal{F}$ containing $a$). Indeed if $f(a) \leq f(b)$, since $f$ is a $\mathbf{P}$-morphism, there exists $b'$ such that $a \leq b'$ and $f(b')=f(b)$, therefore since $a \leq b'$ it is $f^{-1}(f(a)) \leq f^{-1}(f(b))$. The other direction of the implication holds because $f$ is order preserving.
\end{proof}

\begin{definition} \label{defminimalpmorph}
Let $P,Q$ be finite posets and  $f : P \rightarrow Q$ a surjective $\mathbf{P}$-morphism (or equivalently: let $\mathcal{F}$ satisfy conditions \ref{itm:condpartantisym}, \ref{itm:condpartstar} and \ref{itm:condpartincomp} of Proposition \ref{proppartition}). We say that $f$ (or $\mathcal{F}$) is minimal if $\#P=\#Q+1$.
\end{definition}

\begin{remark}
If $\mathcal{F}$ is minimal, then at most one element of $\mathcal{F}$ is not a singleton.
\end{remark}

\begin{theorem} \label{thchain}
Let $f:P \rightarrow Q$ be a surjective $\mathbf{P}$-morphism between finite posets. Let $n = \#P-\#Q$. Then there exist $Q_0, \ldots , Q_n$ with $Q_0=P$, $Q_n=Q$ and $f_i:Q_{i-1} \rightarrow Q_i$ which are minimal surjective $\mathbf{P}$-morphisms for $i=1,\ldots,n$ such that $f= f_n \circ \cdots \circ f_1$. 
\end{theorem}

\begin{proof}
Let $R = \text{dom } f$, we can decompose $f=f'' \circ f'$ where $f'':R \rightarrow Q$ is just the restriction of $f$ on its domain and $f':P \rightarrow R$ is the partial morphism with domain $R$ that acts as the identity on $R$.\\
The morphism $f'':R \rightarrow Q$ is a total morphism\footnote{Since it is a total map, its dual preserves the maximum downset and intersections of downsets. Therefore it is dual to a co-Heyting algebras morphism.}, we prove by induction on $\# R-\#Q$ that it can be decomposed in a chain of minimal surjective $\mathbf{P}$-morphisms.\\
Suppose $\# R-\#Q>1$ and let us consider the partition $\mathcal{F}$ of $R$ given by the fibers of $f''$. Let $x \in P$ be minimal among the elements of $R$ that are not in a singleton of $\mathcal{F}$. Let $G$ be the element of $\mathcal{F}$ containing $x$, then $\#G > 1$ and all the elements of $R$ inside $G$ are incomparable to each other.\\
Let $Q_{n-1}$ be the quotient of $R$ defined by the refining of $\mathcal{F}$ in which $G$ is substituted by $\lbrace x \rbrace$ and $G \backslash \lbrace x \rbrace$, we name this new partition $\mathcal{F}'$.\\
The projection onto the quotient $\pi:R \rightarrow Q_{n-1}$ is a $\mathbf{P}$-morphism because $\mathcal{F}'$ satisfies the conditions \ref{itm:condpartantisym}, \ref{itm:condpartstar} and \ref{itm:condpartincomp} of Proposition \ref{proppartition}. Indeed, it satisfies \ref{itm:condpartantisym} and \ref{itm:condpartincomp} because $\mathcal{F}$ satisfies them and the elements in $G$ are incomparable. To show that \ref{itm:condpartstar} holds it is sufficient to show that for the pairs of sets in $\mathcal{F}'$ in which exactly one of the two is $\lbrace x \rbrace$ or $G \backslash \lbrace x \rbrace$ because $\mathcal{F}$ satisfies \ref{itm:condpartstar} and $\lbrace x \rbrace$ and $G \backslash \lbrace x \rbrace$ are incomparable.\\
Let $A \in \mathcal{F}$ be different from $\lbrace x \rbrace$ and $G \backslash \lbrace x \rbrace$.\\
If $\lbrace x \rbrace \leq A$ then \ref{itm:condpartstar} holds because $\lbrace x \rbrace$ is a singleton.\\
If $A \leq \lbrace x \rbrace$ then there exists $a \in A$ such that $a \leq x$, hence $A$ is a singleton by minimality of $x$, therefore \ref{itm:condpartstar} holds.\\
If $G \backslash \lbrace x \rbrace \leq A$ then we have that $G \leq A$, thus for any $y \in G \backslash \lbrace x \rbrace$ there exists $y' \geq y$ such that $y' \in A$.\\
If $A \leq G \backslash \lbrace x \rbrace$ it is $A \leq G$ thus for any $y \in A$ there exists $y' \geq y$ such that $y'=x$ or $y' \in G \backslash \lbrace x \rbrace$. Suppose there exists $y \in A$ such that there is no $y' \geq y$ such that $y' \in G \backslash \lbrace x \rbrace$: then $x \geq y$, by minimality of $x$ it has to be $A= \lbrace y \rbrace$ then $A \nleq G \backslash \lbrace x \rbrace$, this is absurd.\\
Therefore $\pi:R \rightarrow Q_{n-1}$ is a surjective total $\mathbf{P}$-morphism and we can apply the inductive hypothesis on $\pi$.\\
Then it suffices to show that the order-preserving map $f_n:Q_{n-1} \rightarrow Q$ induced by $f''$ is a $\mathbf{P}$-morphism, because in that case it is obviously surjective and minimal. But this is easy to show because the fibers of $f_n$ are all singletons except one and because $f''$ is a $\mathbf{P}$-morphism.\\
It remains to decompose $f'$, to do that just enumerate the elements of $P \setminus R= \lbrace p_1, \ldots, p_k \rbrace$ and let $f_1 ' :R \cup \lbrace p_1 \rbrace \rightarrow R$ be the partial morphism with domain $R$ that acts as the identity on $R$. Then construct $f_2 ' : R \cup \lbrace p_1,p_2 \rbrace \rightarrow R \cup \lbrace p_1 \rbrace $ in the same way and so on until $p_k$.
\end{proof}

\begin{definition}
We say that a proper extension $L_0 \subseteq L$ of finite CBSes is minimal if there is no intermediate proper extension $L_0 \subsetneq L_1 \subsetneq L$.
\end{definition}

\begin{proposition}
An extension $L_0 \subseteq L$ of finite CBSes is minimal iff the surjective $\mathbf{P}$-morphism that is dual to the inclusion is minimal.
\end{proposition}

\begin{proof}
Let $f : P \rightarrow Q$ be a surjective $\mathbf{P}$-morphism with $\#P=\#Q+1$. And suppose there exist two surjective $\mathbf{P}$-morphisms $g_1 :P \rightarrow R$ and $ g_2 : R \rightarrow Q$ such that $f=g_2 \circ g_1$, being $g_1$ and $g_2$ surjective $\#R$ must be equal to $\#P$ or $\#Q$. In the former case the domain of $g_1$ must be all $P$ and the relative fiber partition could only be the one formed exclusively by singletons because of cardinality, in the latter case the same holds for $g_2$. So either $g_1$ or $g_2$ has to be an isomorphism of posets.\\
Hence if we have two consecutive extensions that form an inclusion whose dual is minimal, then the dual of one of the two extensions is an isomorphism and so the relative extension is the identity.\\
The other implication follows easily from Theorem \ref{thchain}.
\end{proof}

\begin{remark}
By Definition \ref{defminimalpmorph} it follows immediately that there are two different kinds of minimal surjective $\mathbf{P}$-morphisms between finite posets.\\
We call a minimal surjective $\mathbf{P}$-morphism \textit{of the first kind} when there is exactly one element outside its domain and thus the restriction of such map on its domain is bijective and therefore an isomorphism of posets (any bijective $\mathbf{P}$-morphism is an isomorphism of posets). Some of these maps are dual to co-Heyting algebras embeddings but some are not.\\
We call a minimal surjective $\mathbf{P}$-morphism \textit{of the second kind} when it is total, i.e. there are no elements outside its domain, and thus there is exactly a single fiber which is not a singleton and it contains exactly two elements. The maps of the second kind are dual to co-Heyting algebras embeddings.\\
Figures \ref{fig:M1} and \ref{fig:M2} show some examples of minimal surjective $\mathbf{P}$-morphisms and relative extensions of CBSes.\\
We call a finite minimal extension of CBSes either of the first or of the second kind if the corresponding minimal surjective $\mathbf{P}$-morphism is respectively of the first or of the second kind.\\
Therefore, a finite minimal extension of CBSes of the first kind preserves the join-irreducibility of all the join-irreducibles in the domain. Indeed, since the corresponding $\mathbf{P}$-morphism is an isomorphism when restricted on its domain, we have that the downset generated by the preimage of a principal downset is still principal.\\
A finite minimal extension of CBSes of the second kind preserves the join-irreducibility of all the join-irreducibles in the domain except one which becomes the join of the two new join-irreducible elements in the codomain. Indeed, the corresponding $\mathbf{P}$-morphism is total and all its fibers are singletons except one, this implies that the preimage of any principal downset is principal except for one whose preimage is a downset generated by two elements.
\end{remark}

\begin{figure}
\centering

\scalebox{0.6}{
\begin{tikzpicture}[line cap=round,line join=round,>=triangle 45,x=1.0cm,y=1.0cm]
\clip(-10.,-6.) rectangle (10.,1.5);
\fill[fill=black,fill opacity=0.10000000149011612] (-8.4,-0.9) -- (-8.4,-1.9) -- (-7.4,-1.9) -- (-7.4,-0.9) -- cycle;
\fill[fill=black,fill opacity=0.10000000149011612] (-9.404441813761878,-4.090006859162249) -- (-9.404441813761878,-5.090006859162249) -- (-6.404441813761878,-5.090006859162249) -- (-6.404441813761878,-4.090006859162249) -- cycle;
\draw (-6.02,-0.7)-- (-4.02,-0.7);
\draw (-4.52,-0.4)-- (-4.02,-0.7);
\draw (-4.52,-1.)-- (-4.02,-0.7);
\draw (4.228,-0.696)-- (6.228,-0.696);
\draw (5.728,-0.396)-- (6.228,-0.696);
\draw (5.728,-0.996)-- (6.228,-0.696);
\draw (7.73,0.8)-- (7.73,-0.7);
\draw (-3,-0.3) node[anchor=north west] {\begin{huge}
$\emptyset$
\end{huge}};
\draw (-8.3,-1) node[anchor=north west] {\begin{Large}
$\emptyset$
\end{Large}};
\draw (-8.4,-0.9)-- (-8.4,-1.9);
\draw (-8.4,-1.9)-- (-7.4,-1.9);
\draw (-7.4,-1.9)-- (-7.4,-0.9);
\draw (-7.4,-0.9)-- (-8.4,-0.9);
\draw (-5.624441813761879,-4.590006859162249)-- (-3.6244418137618757,-4.590006859162249);
\draw (-4.1244418137618775,-4.290006859162249)-- (-3.6244418137618757,-4.590006859162249);
\draw (-4.1244418137618775,-4.8900068591622485)-- (-3.6244418137618757,-4.590006859162249);
\draw (4.223558186238123,-4.586006859162248)-- (6.223558186238123,-4.586006859162248);
\draw (5.723558186238123,-4.286006859162249)-- (6.223558186238123,-4.586006859162248);
\draw (5.723558186238123,-4.886006859162249)-- (6.223558186238123,-4.586006859162248);
\draw (-9.404441813761878,-4.090006859162249)-- (-9.404441813761878,-5.090006859162249);
\draw (-9.404441813761878,-5.090006859162249)-- (-6.404441813761878,-5.090006859162249);
\draw (-6.404441813761878,-5.090006859162249)-- (-6.404441813761878,-4.090006859162249);
\draw (-6.404441813761878,-4.090006859162249)-- (-9.404441813761878,-4.090006859162249);
\draw (2.4755581862381244,-3.5900068591622487)-- (2.4755581862381244,-5.590006859162249);
\draw (7.046040380934557,-4.591386198152422)-- (8.046040380934558,-5.591386198152422);
\draw (7.046040380934557,-4.591386198152422)-- (8.046040380934558,-3.591386198152424);
\draw (8.046040380934558,-5.591386198152422)-- (9.046040380934556,-4.591386198152422);
\draw (8.046040380934558,-3.591386198152424)-- (9.046040380934556,-4.591386198152422);
\draw [rotate around={0.:(-7.904441813761868,-4.590006859162249)}] (-7.904441813761868,-4.590006859162249) ellipse (1.3080394581011854cm and 0.38531444814568816cm);
\begin{scriptsize}
\draw [fill=black] (2.48,-0.7) circle (2.0pt);
\draw [fill=black] (-7.9,-0.1) circle (2.0pt);
\draw [fill=black] (7.73,-0.7) circle (2.0pt);
\draw [fill=ffffff] (7.73,0.8) circle (2.0pt);
\draw [fill=black] (2.4755581862381244,-5.590006859162249) circle (2.0pt);
\draw [fill=black] (-8.904441813761878,-4.590006859162249) circle (2.0pt);
\draw [fill=black] (-6.904441813761878,-4.590006859162249) circle (2.0pt);
\draw [fill=black] (2.4755581862381244,-3.5900068591622487) circle (2.0pt);
\draw [fill=black] (8.046040380934558,-5.591386198152422) circle (2.0pt);
\draw [fill=black] (8.046040380934558,-3.591386198152424) circle (2.0pt);
\draw [fill=ffffff] (7.046040380934557,-4.591386198152422) circle (2.0pt);
\draw [fill=ffffff] (9.046040380934556,-4.591386198152422) circle (2.0pt);
\draw [fill=black] (-2.261696854106027,-4.586989587362458) circle (2.0pt);
\end{scriptsize}
\end{tikzpicture}
}

\caption{Simplest examples of minimal extensions and their duals; on the left we show the surjective $\mathbf{P}$-morphisms and on the right the corresponding minimal extensions of CBSes. The domain is denoted by a rectangle and the partition into fibers is represented by the encircled sets of points. The white points represents the elements outside the images of the inclusions. Notice that the inclusion on the top is not a co-Heyting algebras morphism.} \label{fig:M1}
\end{figure}
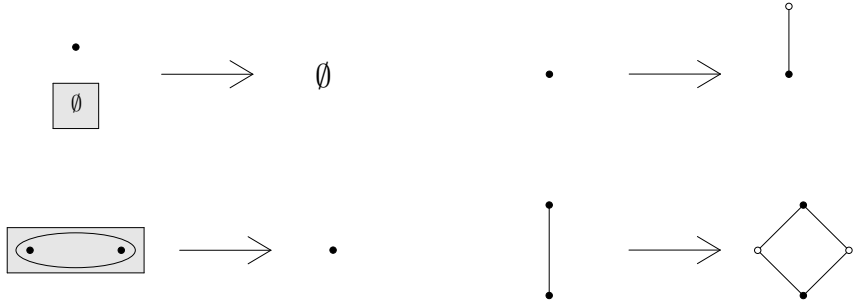

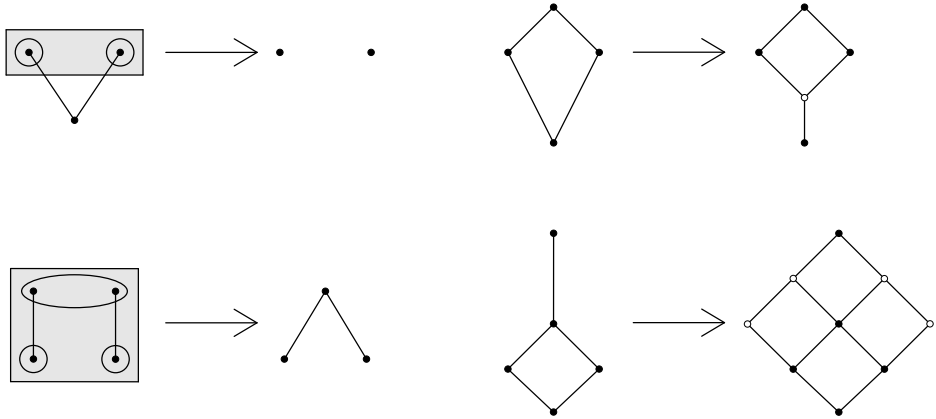
\begin{figure}
\centering

\scalebox{0.6}{
\begin{tikzpicture}[line cap=round,line join=round,>=triangle 45,x=1.0cm,y=1.0cm]
\clip(-10.,-7.5) rectangle (11.5,2.5);
\fill[line width=0.8pt,fill=black,fill opacity=0.10000000149011612] (-9.25,1.5) -- (-9.25,0.5) -- (-6.25,0.5) -- (-6.25,1.5) -- cycle;
\fill[line width=0.8pt,fill=black,fill opacity=0.10000000149011612] (-9.152008196161473,-6.276651905297075) -- (-6.351863437958693,-6.279893658890327) -- (-6.351863437958693,-3.7798936588903276) -- (-9.152008196161473,-3.7766519052970793) -- cycle;
\draw [line width=0.8pt] (-8.75,1.)-- (-7.75,-0.5);
\draw [line width=0.8pt] (-6.75,1.)-- (-7.75,-0.5);
\draw [line width=0.8pt] (-8.75,1.) circle (0.3cm);
\draw [line width=0.8pt] (-6.75,1.) circle (0.3cm);
\draw [line width=0.8pt] (-9.25,1.5)-- (-9.25,0.5);
\draw [line width=0.8pt] (-9.25,0.5)-- (-6.25,0.5);
\draw [line width=0.8pt] (-6.25,0.5)-- (-6.25,1.5);
\draw [line width=0.8pt] (-6.25,1.5)-- (-9.25,1.5);
\draw [line width=0.8pt] (-5.75,1.)-- (-3.75,1.);
\draw [line width=0.8pt] (-4.25,1.3)-- (-3.75,1.);
\draw [line width=0.8pt] (-4.25,0.7)-- (-3.75,1.);
\draw [line width=0.8pt] (1.75,1.)-- (2.75,-1.);
\draw [line width=0.8pt] (1.75,1.)-- (2.75,2.);
\draw [line width=0.8pt] (2.75,-1.)-- (3.75,1.);
\draw [line width=0.8pt] (2.75,2.)-- (3.75,1.);
\draw [line width=0.8pt] (7.25,1.)-- (8.25,0.);
\draw [line width=0.8pt] (7.25,1.)-- (8.25,2.);
\draw [line width=0.8pt] (8.25,0.)-- (9.25,1.);
\draw [line width=0.8pt] (8.25,2.)-- (9.25,1.);
\draw [line width=0.8pt] (4.498,1.004)-- (6.498,1.004);
\draw [line width=0.8pt] (5.998,1.304)-- (6.498,1.004);
\draw [line width=0.8pt] (5.998,0.704)-- (6.498,1.004);
\draw [line width=0.8pt] (8.25,0.)-- (8.25,-1.);
\draw [line width=0.8pt] (-5.7520081961614755,-4.976651905297078)-- (-3.7520081961614755,-4.976651905297078);
\draw [line width=0.8pt] (-4.2520081961614755,-4.676651905297077)-- (-3.7520081961614755,-4.976651905297078);
\draw [line width=0.8pt] (-4.2520081961614755,-5.276651905297077)-- (-3.7520081961614755,-4.976651905297078);
\draw [line width=0.8pt] (4.495991803838525,-4.972651905297078)-- (6.495991803838525,-4.972651905297078);
\draw [line width=0.8pt] (5.995991803838525,-4.672651905297077)-- (6.495991803838525,-4.972651905297078);
\draw [line width=0.8pt] (5.995991803838525,-5.272651905297077)-- (6.495991803838525,-4.972651905297078);
\draw [line width=0.8pt] (-8.652008196161473,-4.276651905297079)-- (-8.652008196161473,-5.776651905297075);
\draw [line width=0.8pt] (-8.652008196161473,-5.776651905297075) circle (0.3cm);
\draw [line width=0.8pt] (-6.851863437958693,-4.279893658890328)-- (-6.851863437958693,-5.779893658890327);
\draw [line width=0.8pt] (-6.851863437958693,-5.779893658890327) circle (0.3cm);
\draw [line width=0.8pt] (-9.152008196161473,-6.276651905297075)-- (-6.351863437958693,-6.279893658890327);
\draw [line width=0.8pt] (-6.351863437958693,-6.279893658890327)-- (-6.351863437958693,-3.7798936588903276);
\draw [line width=0.8pt] (-6.351863437958693,-3.7798936588903276)-- (-9.152008196161473,-3.7766519052970793);
\draw [line width=0.8pt] (-9.152008196161473,-3.7766519052970793)-- (-9.152008196161473,-6.276651905297075);
\draw [rotate around={-0.08442111093369144:(-7.75193581706012,-4.278272782093703)},line width=0.8pt] (-7.75193581706012,-4.278272782093703) ellipse (1.1586601617252432cm and 0.36377397359568864cm);
\draw [line width=0.8pt] (-3.1514222724638246,-5.778347451320107)-- (-2.2514222724638246,-4.278347451320109);
\draw [line width=0.8pt] (-2.2514222724638246,-4.278347451320109)-- (-1.3514222724638245,-5.778347451320107);
\draw [line width=0.8pt] (1.7485777275361774,-5.998347451320105)-- (2.7485777275361776,-6.9483474513201005);
\draw [line width=0.8pt] (1.7485777275361774,-5.998347451320105)-- (2.7485777275361776,-4.998347451320109);
\draw [line width=0.8pt] (2.7485777275361776,-6.9483474513201005)-- (3.7485777275361785,-5.998347451320105);
\draw [line width=0.8pt] (2.7485777275361776,-4.998347451320109)-- (3.7485777275361785,-5.998347451320105);
\draw [line width=0.8pt] (2.749871987268069,-2.9972259343359875)-- (2.7485777275361776,-4.998347451320109);
\draw [line width=0.8pt] (8.001826238432699,-6.00065959666554)-- (9.001826238432695,-6.950659596665533);
\draw [line width=0.8pt] (8.001826238432699,-6.00065959666554)-- (9.001826238432695,-5.000659596665549);
\draw [line width=0.8pt] (9.001826238432695,-6.950659596665533)-- (10.001826238432693,-6.00065959666554);
\draw [line width=0.8pt] (9.001826238432695,-5.000659596665549)-- (10.001826238432693,-6.00065959666554);
\draw [line width=0.8pt] (7.000233773781643,-5.0031178272561725)-- (8.001826238432699,-6.00065959666554);
\draw [line width=0.8pt] (10.001826238432693,-6.00065959666554)-- (11.000628684561091,-4.999369725865691);
\draw [line width=0.8pt] (7.000233773781643,-5.0031178272561725)-- (8.002798326505205,-3.9967444176376037);
\draw [line width=0.8pt] (8.002798326505205,-3.9967444176376037)-- (9.001826238432695,-5.000659596665549);
\draw [line width=0.8pt] (9.001826238432695,-5.000659596665549)-- (10.,-4.);
\draw [line width=0.8pt] (10.,-4.)-- (11.000628684561091,-4.999369725865691);
\draw [line width=0.8pt] (10.,-4.)-- (9.003120498164586,-2.9995380796814284);
\draw [line width=0.8pt] (8.002798326505205,-3.9967444176376037)-- (9.003120498164586,-2.9995380796814284);
\begin{scriptsize}
\draw [fill=black] (-8.75,1.) circle (2.0pt);
\draw [fill=black] (-6.75,1.) circle (2.0pt);
\draw [fill=black] (-7.75,-0.5) circle (2.0pt);
\draw [fill=black] (-3.25,1.) circle (2.0pt);
\draw [fill=black] (-1.25,1.) circle (2.0pt);
\draw [fill=black] (2.75,-1.) circle (2.0pt);
\draw [fill=black] (2.75,2.) circle (2.0pt);
\draw [fill=black] (1.75,1.) circle (2.0pt);
\draw [fill=black] (3.75,1.) circle (2.0pt);
\draw [fill=ffffff] (8.25,0.) circle (2.0pt);
\draw [fill=black] (8.25,2.) circle (2.0pt);
\draw [fill=black] (7.25,1.) circle (2.0pt);
\draw [fill=black] (9.25,1.) circle (2.0pt);
\draw [fill=black] (8.25,-1.) circle (2.0pt);
\draw [fill=black] (-8.652008196161473,-5.776651905297075) circle (2.0pt);
\draw [fill=black] (-8.652008196161473,-4.276651905297079) circle (2.0pt);
\draw [fill=black] (-6.851863437958693,-5.779893658890327) circle (2.0pt);
\draw [fill=black] (-6.851863437958693,-4.279893658890328) circle (2.0pt);
\draw [fill=black] (-2.2514222724638246,-4.278347451320109) circle (2.0pt);
\draw [fill=black] (-3.1514222724638246,-5.778347451320107) circle (2.0pt);
\draw [fill=black] (-1.3514222724638245,-5.778347451320107) circle (2.0pt);
\draw [fill=black] (2.7485777275361776,-6.9483474513201005) circle (2.0pt);
\draw [fill=black] (2.7485777275361776,-4.998347451320109) circle (2.0pt);
\draw [fill=black] (1.7485777275361774,-5.998347451320105) circle (2.0pt);
\draw [fill=black] (3.7485777275361785,-5.998347451320105) circle (2.0pt);
\draw [fill=black] (2.749871987268069,-2.9972259343359875) circle (2.0pt);
\draw [fill=black] (9.001826238432695,-6.950659596665533) circle (2.0pt);
\draw [fill=black] (9.001826238432695,-5.000659596665549) circle (2.0pt);
\draw [fill=black] (8.001826238432699,-6.00065959666554) circle (2.0pt);
\draw [fill=black] (10.001826238432693,-6.00065959666554) circle (2.0pt);
\draw [fill=black] (9.003120498164586,-2.9995380796814284) circle (2.0pt);
\draw [fill=ffffff] (7.000233773781643,-5.0031178272561725) circle (2.0pt);
\draw [fill=ffffff] (11.000628684561091,-4.999369725865691) circle (2.0pt);
\draw [fill=ffffff] (10.,-4.) circle (2.0pt);
\draw [fill=ffffff] (8.002798326505205,-3.9967444176376037) circle (2.0pt);
\end{scriptsize}
\end{tikzpicture}
}

\caption{More complex examples of minimal extensions and their duals.} \label{fig:M2}
\end{figure}

It turns out that we can characterize the finite minimal extensions of CBSes by means of their generators.

\begin{definition} \label{defprim1}
Let $L_0$ be a finite CBS and $L$ an extension of $L_0$. We call an element $x \in L$ \textit{primitive of the first kind} over $L_0$ if the following conditions are satisfied:
\begin{enumerate}
\item $x \notin L_0$ \label{itm:prim1cond1}
\end{enumerate}
and for any $a$ join-irreducible of $L_0$:
\begin{enumerate}[resume]
\item $a-x \in L_0$, \label{itm:prim1cond2}
\item $x-a = x \text{ or } x-a = 0$. \label{itm:prim1cond3}
\end{enumerate}
\end{definition}

\begin{theorem} \label{theoext1}
Let $L_0$ be a finite CBS and $L$ an extension of $L_0$.\footnote{Notice that we do not require $L$ to be a finite CBS.}
If $x \in L$ is primitive of the first kind over $L_0$ then the sub-CBS $L_0 \langle x \rangle$ of $L$ generated by $x$ over $L_0$ is a finite minimal extension of $L_0$ of the first kind.
\end{theorem}

Before proving Theorem \ref{theoext1} we need the following lemma:

\begin{lemma} \label{lemmaprim1}
Let $L_0$ be a finite CBS, $L$ an extension of $L_0$ and $x \in L$ primitive of the first kind over $L_0$, then the two following properties hold:
\begin{enumerate}[label=(\roman*)] 
\item $\forall a \in L_0 \quad a-x \in L_0$, \label{itm:prim1condi}
\item $\forall a \in L_0 \quad x-a = x \text{ or } x-a = 0$. \label{itm:prim1condii}
\end{enumerate}
\end{lemma}

\begin{proof}
Let $a \in L_0$ and $a_1, \ldots, a_n$ be its join-irreducible components in $L_0$, since $L_0$ is finite we have $a=a_1 \vee \cdots \vee a_n$. To prove \ref{itm:prim1condi} observe that
\[
a-x=(a_1-x) \vee \cdots \vee (a_n-x)
\]
which is an element of $L_0$ because it is join of elements of $L_0$ as a consequence of \ref{itm:prim1cond2} of Definition \ref{defprim1}.\\
Furthermore to prove \ref{itm:prim1condii} notice that
\[
x-a=x-(a_1 \vee \cdots \vee a_n)=((x-a_1)- \cdots )-a_n
\]
and that \ref{itm:prim1cond3} of Definition \ref{defprim1} implies that there are two possibilities: $x-a_i=x$ for any $i=1, \ldots,n$ or $x-a_i=0$ for some $i$. In the former case we have $x-a=x$, in the latter suppose that $i$ is the smallest index such that  $x-a_i=0$ then 
\[
x-a=((x-a_i)- \cdots )-a_n=(0- \cdots )-a_n=0.
\]
\end{proof}

\begin{proof}[Proof of Theorem \ref{theoext1}]
Let $L'$ be the sub $\vee$-semilattice of $L$ generated by $x$ over $L_0$, we show that $L'$ actually coincides with $L_0 \langle x \rangle$.\\
$L'$ is clearly finite, its elements are the elements of $L_0$ and the elements of the form $a \vee x$ with $a \in L_0$. It follows from \ref{itm:prim1condi} and \ref{itm:prim1condii} of Lemma \ref{lemmaprim1} that if 
 $a,b,c,d \in L_0 \cup \lbrace x \rbrace $ then $(a \vee b) -(c \vee d) = (a-(c \vee d) ) \vee (b-(c \vee d) )= ( (a-c )-d) \vee ((b-c)-d)$ belong to $L'$. Therefore $L'=L_0 \langle x \rangle$.\\
We want to show that the join-irreducibles of $L_0 \langle x \rangle$ are exactly the join-irreducibles of $L_0$ and $x$.\\
$x$ is a join-irreducible element of $L_0 \langle x \rangle$, indeed $x \neq 0$ since by hypothesis $x \notin L_0$ and suppose that $x \leq a \vee b$ with $a,b \in L_0 \langle x \rangle$ and $a,b \ngeq x$; therefore $a$ and $b$ must be elements of $L_0$ because they cannot be of the form $c \vee x$ with $c \in L_0$. It follows from \ref{itm:prim1condii} of Lemma \ref{lemmaprim1} and $a,b \ngeq x$ that $x-a = x-b = x$ and so $0=x-(a \vee b )=  ( x-b)-a= x-a = x$, this is absurd because $x \neq 0$.\\
The join-irreducible elements of $L_0$ are still join-irreducible in $L_0 \langle x \rangle$. It is sufficient to show that for any $g$ join-irreducible in $L_0$ if $g \leq a \vee x $ with $a \in L_0$ then $g \leq a$ or $g \leq x$. Notice that being $L$ a CBS it is $g= (g-x ) \vee ( g-(g-x))$ (see Remark \ref{listsimplefacts}), we also have by \ref{itm:prim1cond2} of Definition \ref{defprim1} that $g-x$ and $ g-(g-x)$ are in $L_0$. Then being $g$ join-irreducible in $L_0$ we get $g=g-x$ or $g =   g-(g-x)$. In the latter case $g-x= g-(g-(g-x))= g-g = 0$ so $g \leq x$. In the former case $0=g- (a \vee x )  = (g-x)-a= g-a$ so $g \leq a$.\\
Clearly if an element of the form $x \vee a$ with $a \in L_0$ is different from $a$ and $x$ it cannot be join-irreducible in $L_0 \langle x \rangle$. Also if an element of $L_0$ is not join-irreducible in $L_0$ it cannot be join-irreducible in $L_0 \langle x \rangle$. Hence the join-irreducible elements of $L_0 \langle x \rangle$ are exactly the join-irreducible elements of $L_0$ and $x$.\\
Therefore the extension $L_0 \hookrightarrow L_0 \langle x \rangle$ is minimal since $L_0 \langle x \rangle$ contains exactly one join-irreducible element more than $L_0$. Notice that $L_0 \langle x \rangle$ is a minimal extension of $L_0$ of the first kind because the join-irreducibility of all the join-irreducibles of $L_0$ is preserved.
\end{proof}

\begin{definition} \label{defprim2}
Let $L_0$ be a finite CBS and $L$ an extension of $L_0$.\footnote{Again we do not require $L$ to be a finite CBS.}
We call a couple of elements $(x_1,x_2) \in L^2$ \textit{primitive of the second kind} over $L_0$ if the following conditions are satisfied:
\begin{enumerate}
\item $x_1,x_2 \notin L_0$ and $x_1 \neq x_2$ \label{itm:prim2cond1}
\end{enumerate}
and there exists $g$ join-irreducible element of $L_0$ such that:
\begin{enumerate}[resume]
\item $g-x_1=x_2$ and $g-x_2=x_1$, \label{itm:prim2cond2}
\item for any join-irreducible element $a$ of $L_0$ such that $a <g$ we have $a-x_i \in L_0$ for $i=1,2$. \label{itm:prim2cond3}
\end{enumerate}
\end{definition}

\begin{remark} \label{remarkprim2}
$g$ in Definition \ref{defprim2} is univocally determined by $(x_1,x_2)$ since $g= x_1 \vee x_2$.\\
Indeed, by property \ref{itm:prim2cond2} of Definition \ref{defprim2} we have $x_1 \leq g$, $x_2 \leq g$ and also $g-(x_1 \vee x_2)=(g-x_1)-x_2=x_2-x_2=0$ that implies $g \leq x_1 \vee x_2$.
\end{remark}

\begin{theorem} \label{theoext2}
Let $L_0$ be a finite CBS and $L$ an extension of $L_0$.
If $(x_1,x_2) \in L^2$ is primitive of the second kind over $L_0$ then the sub-CBS $L_0 \langle x_1, x_2 \rangle$ of $L$ generated by $\lbrace x_1,x_2 \rbrace$ over $L_0$ is a finite minimal extension of $L_0$ of the second kind.
\end{theorem}

Before proving Theorem \ref{theoext2} we need the following lemma:

\begin{lemma} \label{lemmaprim2}
Let $L_0$ be a finite CBS, $L$ an extension of $L_0$ and $(x_1,x_2) \in L^2$ primitive of the second kind over $L_0$, then the two following properties hold:
\begin{enumerate}[label=(\roman*)] 
\item $\forall a \in L_0 \quad a-x_i \in L_0 \text{ or } a-x_i = b \vee x_j$ with $b \in L_0$ for $\lbrace i,j \rbrace = \lbrace 1,2 \rbrace$. \label{itm:prim2condi}
\item $\forall a \in L_0 \quad  x_i-a = x_i \text{ or } x_i-a = 0$ for $i=1,2$. \label{itm:prim2condii}
\end{enumerate}
\end{lemma}

\begin{proof}
To show \ref{itm:prim2condi} we first prove that if $a \neq g$ is join-irreducible in $L_0$, then $a-x_i \in L_0$. If $a<g$ this is covered by the hypothesis \ref{itm:prim2cond3} of Definition \ref{defprim2}. Now suppose that $a$ is a join-irreducible element of $L_0$ such that $a \nleq g$ then $a-g=a$ because $a$ is join-irreducible. Thus $a=a-g \leq a-x_i \leq a$ since $x_i \leq g$ (because $x_i=g-x_j \leq g$ with $i \neq j$) and thus $a-x_i=a \in L_0$ for $i=1,2$.
We now prove \ref{itm:prim2condi} for all $a \in L_0$.\\
Let $a \in L_0$ and $a_1, \ldots, a_n$ be its join-irreducible components in $L_0$, since $L_0$ is finite we have $a=a_1 \vee \cdots \vee a_n$. To prove \ref{itm:prim2condi} we consider two cases: $g$ is a join-irreducible component of $a$ or $g$ is not a join-irreducible component of $a$. In the former case, when $g$ is a join-irreducible component of $a$, suppose $a_1=g$, then
\[
a-x_i=(g-x_i) \vee \cdots \vee (a_n-x_i)=x_j \vee (a_2-x_i) \vee \cdots \vee (a_n-x_i)
\]
with $\lbrace i,j \rbrace = \lbrace 1,2 \rbrace$, notice that $(a_2-x_i) \vee \cdots \vee (a_n-x_i) \in L_0$ because it is join of elements of $L_0$ by what we have just proved. In the latter case, $g$ is not a join-irreducible component of $a$, we have 
\[
a-x=(a_1-x) \vee \cdots \vee (a_n-x)
\]
which is an element of $L_0$ because it is join of elements of $L_0$ as a consequence of what we have just proved.\\
Furthermore, to prove \ref{itm:prim2condii} notice that since $g$ is join-irreducible in $L_0$ we have that for any $a \in L_0$ there are two cases to consider: $g \leq a$ or $g-a=g$. In the former case we have, since $x_i \leq g$ by \ref{itm:prim2cond2} of Definition \ref{defprim2}, that $x_i-a=0$ for $i=1,2$ because $x_i \leq g \leq a$. In the latter case, since $g-a=g$, we have
\[
x_i-a=(g-x_j)-a=(g-a)-x_j=g-x_j=x_i
\]
for $\lbrace i,j \rbrace = \lbrace 1,2 \rbrace$.\\
\end{proof}

\begin{proof}[Proof of Theorem \ref{theoext2}]
Let $L'$ be the sub $\vee$-semilattice of $L$ generated by $\lbrace x_1,x_2 \rbrace$ over $L_0$.\\
As shown in Remark \ref{remarkprim2} we have $g= x_1 \vee x_2$. Also $x_2 - x_1 = x_2$ and $x_1 - x_2 = x_1$. Indeed $x_2 - x_1= (g - x_1)  -x_1 =  g-x_1= x_2$, the other case is symmetrical.\\
Hence by reasoning in a similar way as in the proof of Theorem \ref{theoext1}, using properties \ref{itm:prim2condi} and \ref{itm:prim2condii} of Lemma \ref{lemmaprim2}, we get that $L'=L_0 \langle x_1, x_2 \rangle$.\\
We now want to show that the join-irreducibles of $L_0 \langle x_1, x_2 \rangle$ are exactly $x_1,x_2$ and the join-irreducibles of $L_0$ different from $g$.\\
First, notice that if an element of $L_0$ is not join-irreducible in $L_0$ it cannot be join-irreducible in $L_0 \langle x_1, x_2 \rangle$. Furthermore, the only elements of $L_0 \langle x_1, x_2 \rangle$ not in $L_0$ that could be join-irreducible in $L_0 \langle x_1, x_2 \rangle$ are $x_1,x_2$ because $L_0 \langle x_1, x_2 \rangle$ is the $\vee$-semilattice generated by $\lbrace x_1,x_2 \rbrace$ over $L_0$.\\
We now show that $x_1,x_2$ are join-irreducible in $L_0 \langle x_1, x_2 \rangle$.\\
Suppose $x_1$ is not join-irreducible in $L_0 \langle x_1,x_2 \rangle$ and let $y_1, \ldots, y_r$ be its join-irreducible components. One of them must be $x_2$ because $x_1 \notin L_0$ and we observed that all the join-irreducible elements of $L_0 \langle x_1,x_2 \rangle$ are in $L_0 \cup \lbrace x_1,x_2 \rbrace$. But then $x_2 \leq x_1$ and therefore, by what was shown above, $0=x_2-x_1=x_2$ which is absurd because $x_2 \notin L_0$.
The same reasoning holds for the join-irreducibility of $x_1$.\\
It remains to show that the only element join-irreducible of $L_0$ which is not join-irreducible in $L_0 \langle x_1, x_2 \rangle$ is $g$.\\
Observe that $g$ is not join-irreducible in $L_0 \langle x_1, x_2 \rangle$ because $g=x_1 \vee x_2$ and $x_1,x_2 \neq g$ since $x_1,x_2 \notin L_0$.\\
Let $b \in L_0$ be join-irreducible in $L_0$ but not in $L_0 \langle x_1, x_2 \rangle$, let $y_1, \ldots, y_r$ be the join-irreducible components of $b$ in $L_0 \langle x_1, x_2 \rangle$. From what we observed above it follows that the $y_i$'s are in $L_0 \cup \lbrace x_1,x_2 \rbrace$ and since $b$ is join-irreducible in $L_0$ at least one of them is not in $L_0$. We can suppose $y_1=x_1$, so $x_1 \leq b$. This implies that $g \leq b$, indeed one among $y_2, \ldots, y_r$ has to be $x_2$ because otherwise $y_2 \vee \cdots \vee y_r \in L_0$ and being the $y_i$'s the join-irreducible components of $b$ we have that $x_1= b-(y_2 \vee \cdots \vee y_r ) $ must be in $L_0$, this is absurd. If $g < b$ then $b-g= b$ because $b$ is join-irreducible in $L_0$ but in this case $x_1=y_1 \leq b=b-g \leq b-x_1 = y_2 \vee \cdots \vee y_r$ and this is not possible because the $y_i$'s  are the join-irreducible components of $b$. This implies $b=g$.\\
Therefore the extension $L_0 \hookrightarrow L_0 \langle x_1,x_2 \rangle$ is minimal since the number of join-irreducibles of $L_0 \langle x_1,x_2 \rangle$ is greater by one than the number of the join-irreducibles of $L_0$.\\
Notice that $L_0 \langle x_1,x_2 \rangle$ is a minimal extension of $L_0$ of the second kind because the join-irreducibility of all but one of the join-irreducibles of $L_0$ is preserved.
\end{proof}

\begin{theorem} \label{theominimalprimitive}
Let $L_0$ be a finite CBS and $L$ a finite minimal extension of $L_0$, then $L$ is generated over $L_0$ either by a primitive element $x \in L$ of the first kind over $L_0$ or by $x_1,x_2 \in L$ forming a primitive couple $(x_1,x_2)$ of the second kind over $L_0$.
\end{theorem}

\begin{proof}
Let $f : P \rightarrow Q$ be the surjective minimal $\mathbf{P}$-morphism dual to the inclusion of $L_0$ into $L$. Recall that $P$ and $Q$ are respectively the posets of the join-irreducible elements of $L$ and $L_0$.\\
We consider two cases:\\
The first case is when $f$ is of the first kind, i.e. $\text{dom } f \neq P$ and there exists only one element $p \in P \setminus \text{dom } f$. In this case, by minimality of $f$, the restriction of $f$ on its domain is an isomorphism of posets. We want to prove that $x=p$ is a primitive element of $L$ of the first kind over $L_0$.\\
We observe that the downset $\fregiu p$ cannot be generated by the preimage of any downset in $Q$ because $p$ is not in the domain of $f$, therefore $x \notin L_0$.\\
For any $q \in Q$ let $q'$ be the unique element of $P$ in the preimage of $q$ by $f$, then $\fregiu f^{-1} (\fregiu q)=\fregiu q'$ because $f$ is a $\mathbf{P}$-morphism. Hence if $q' \leq p$ then $\fregiu q'-\fregiu p= \emptyset$ and if $q' \nleq p$ then $\fregiu q' -\fregiu p=\fregiu q'$. This translates to the fact that for any $a$ join-irreducible of $L_0$ we have $a-x \in L_0$ because both $\emptyset$ and $\fregiu q'$ are generated by the preimage of a downset of $Q$. Furthermore, for any $q \in Q$ if $p \leq q'$ then $\fregiu p - \fregiu q'=\emptyset$ and if $p \nleq q'$ then $\fregiu p-\fregiu q'=\fregiu p$. Thus for any $a$ join-irreducible of $L_0$ either $x-a=0$ or $x-a=x$.\\
The second case is when $f$ is of the second kind, i.e. $\text{dom } f = P$ and only two elements $p_1,p_2$ have the same image by $f$, recall that $p_1,p_2$ are incomparable. We want to prove that $x_1=p_1$ and $x_2=p_2$ form a primitive couple of elements of $L$ of the second kind over $L_0$.\\
$x_1 \neq x_2$ and $x_1,x_2 \in L_0$ because the downsets $\fregiu p_1$ and $\fregiu p_2$ are distinct and neither of them is generated by the preimage of a downset in $Q$. Indeed, since $f$ is a total map the preimages of downsets of $Q$ are already downsets of $P$ and any preimage contains $p_1$ iff it contains $p_2$.\\
Let $f(p_1)=f(p_2)=g \in Q$ then $f^{-1}(g)= \lbrace p_1, p_2 \rbrace$, since $f$ is total $f^{-1}(\fregiu g)$ is a downset and thus $f^{-1}(\fregiu g)= \fregiu f^{-1}(\fregiu g)$. We have that $f^{-1}(\fregiu g)=\fregiu p_1 \cup \fregiu p_2$ because $f$ is a $\mathbf{P}$-morphism and therefore $f^{-1}(\fregiu g)-\fregiu p_1=\fregiu p_2$ and $f^{-1}(\fregiu g)-\fregiu p_2=\fregiu p_1$ because $p_1$ and $p_2$ are incomparable. Therefore $g-x_1=x_2$ and $g-x_2=x_1$.\\
Let $q \in Q$ such that $q <g$ and $q'$ be the unique element of $P$ in the preimage of $q$ by $f$, then $\fregiu f^{-1} (\fregiu q)=\fregiu q'$ because $f$ is a $\mathbf{P}$-morphism. Hence if $q' \leq p$ then $\fregiu q'-\fregiu p= \emptyset$ and if $q' \nleq p$ then $\fregiu q' -\fregiu p=\fregiu q'$. Both $\emptyset$ and $\fregiu q'$ are generated by the preimage of a downset of $Q$. This means that for any $a$ join-irreducible of $L_0$ such that $a < g$ we have $a-x_i \in L_0$ for $i=1,2$.
\end{proof}

\begin{definition}
Let $L_0$ be a finite CBS.\\
We call \textit{signature of the first kind} in $L_0$ a couple $(h,G)$ where $h \in L_0$ and $G$ is a set of two-by-two incomparable join-irreducible elements of $L_0$ such that $h < g$ for all $g \in G$. We allow $G$ to be empty.\\
We call \textit{signature of the second kind} in $L_0$ a triple $(h_1,h_2,g)$ where $h_1,h_2 \in L_0$, $g$ is a join-irreducible element of $L_0$ such that $h_1 \vee h_2 =g^-$ the unique predecessor of $g$ in $L_0$.
\end{definition}

\begin{theorem} \label{theosign}
Let $L_0$ be a finite CBS. To give a minimal finite extension either of the first or of the second kind of $L_0$ (up to isomorphism over $L_0$) is equivalent to give respectively:
\begin{enumerate}
\item A signature $(h,G)$ of the first kind in $L_0$.
\item A signature $(h_1,h_2,g)$ of the second kind in $L_0$.
\end{enumerate} 
\end{theorem}

Once again finite duality shows its usefulness. Indeed, we have the following lemma:

\begin{lemma} \label{lemmasign}
Let $Q$ be a finite poset. To give a minimal surjective $\mathbf{P}$-morphism $f$ with codomain $Q$ either of the first or of the second kind (up to isomorphism) is equivalent to give respectively:
\begin{enumerate}
\item $D,U$ respectively a downset and an upset\footnotemark of $Q$ such that $D \cap U= \emptyset$ and for any $d \in D, u \in U$ we have $d \leq u$. \label{itm:lemmasign1}
\item $g \in Q$ and $D_1,D_2$ downsets of $Q$ such that $D_1 \cup D_2= \fregiu g \setminus \lbrace g \rbrace$. \label{itm:lemmasign2}
\end{enumerate}
\footnotetext{The definitions of upset and of the upset $\fresu a$ generated by an element $a$ are analogous to the definitions for the downsets replacing $\leq$ with $\geq$.}
\end{lemma}

\begin{proof}
Let $f:P \to Q$ be a minimal surjective $\mathbf{P}$-morphism.\\
If $f$ is of the first kind and $\text{dom } f = P \setminus \lbrace x \rbrace$ take $D=f(\fregiu x \setminus \lbrace x \rbrace)$ and $U=f(\fresu x \setminus \lbrace x \rbrace)$.\\
If $f$ is of the second kind, i.e. $\text{dom } f = P$, then there is exactly one $g \in Q$ such that $f^{-1}(g)=\lbrace x_1,x_2 \rbrace$ consisting of two elements of $P$. Take $D_i= f(\fregiu x_i \setminus \lbrace x_i \rbrace)$ for $i=1,2$.\\
On the other hand, given $D,U$ as in \ref{itm:lemmasign1} ,we obtain a minimal surjective $\mathbf{P}$-morphism $f:P \to Q$ by taking $P= Q \sqcup \lbrace x \rbrace$ and extending the order of $Q$ setting $q < x$ iff $q \in D$ and $x < q$ iff $q \in U$ for any $q \in Q$. Take $\text{dom } f= Q \subset P$ and $f$ as the identity on its domain.\\
Given $g \in Q$, $D_1,D_2$ as in \ref{itm:lemmasign2} obtain a minimal surjective $\mathbf{P}$-morphism $f:P \to Q$ taking $P= Q \setminus \lbrace g \rbrace \sqcup \lbrace x_1,x_2 \rbrace$ and extending the order of $Q \setminus \lbrace g \rbrace$ setting $q < x_i$ iff $q \in D_i$ and $x_i < q$ iff $g < q$ for any $q \in Q$. Take $\text{dom } f= P$ and $f$ maps $x_1,x_2$ into $g$ and acts as the identity on $Q \setminus \lbrace g \rbrace$.\\
Now let $f_1 : P_1 \rightarrow Q$ and $f_2 : P_2 \rightarrow Q$ be two surjective $\mathbf{P}$-morphisms to which are associated the same $(D,U)$ or $(D_1,D_2,g)$, we show that there exists an isomorphism of posets $\varphi :P_1 \rightarrow P_2$ such that $f_2 \circ \varphi= f_1$.\\
Suppose $f_1,f_2$ are of the first kind and the same $(D,U)$ is associated to both of them. Then $\text{dom } f_1 = P_1 \setminus \lbrace p_1 \rbrace $ and $ \text{dom } f_2 = P_2 \setminus \lbrace p_2 \rbrace $.
Being $f_1,f_2$ two $\mathbf{P}$-morphisms which are isomorphisms when restricted on their domains, we can invert the restriction of $f_2$ and compose it with the restriction of $f_1$ to obtain an isomorphism of posets $\varphi' : \text{dom } f_1 \rightarrow \text{dom } f_2$.\\
It remains to extend $\varphi'$ to an isomorphism $\varphi :P_1 \rightarrow P_2$, just set $\varphi (p_1)=p_2$; $\varphi$ so defined is an isomorphism of posets, we need to show that it reflects and preserves the order of $P_1$. $f_1$ and $f_2$ map respectively the elements smaller than $p_1$ and $p_2$ into the same elements of $Q$ and the elements greater than $p_1$ and $p_2$ into the same elements of $Q$ by hypothesis. Hence $\varphi'$ maps the elements smaller than $p_1$ into the elements smaller than $p_2$ and the elements greater than $p_1$ into the elements greater than $p_2$ and so does its inverse. It follows that $\varphi$ is an isomorphism of posets.\\
Suppose $f_1,f_2$ are of the second kind and the same $(D_1,D_2,g)$ is associated to both of them. Then $f_1,f_2$ are total, i.e. $\text{dom } f_1=P_1$ and $ \text{dom } f_2=P_2$.
The orders restricted on $P_1 \backslash f_1^{-1}(g)$ and $P_2 \backslash f_2^{-1}(g)$ are both isomorphic to $Q \backslash \lbrace g \rbrace$ with isomorphisms given by the restrictions of $f_1,f_2$, indeed given two elements $a,b \in P_1 \backslash f_i^{-1}(g)$ it is $f_i(a) \leq f_i(b)$ iff $a \leq b$ because $f_i$ is a $\mathbf{P}$-morphism for $i=1,2$. Composing these two isomorphisms we obtain an isomorphism $\varphi' : P_1 \backslash f_1^{-1}(g) \rightarrow P_2 \backslash f_2^{-1}(g)$.\\
We now extend it to $\varphi:P_1 \to P_2$.\\
Let $f_i^{-1}(g)= \lbrace x_{1,i}, x_{2,i} \rbrace$ for $i=1,2$, we can suppose to have ordered the indices in such a way that $f_i(\fregiu x_{j,i} \backslash \lbrace x_{j,i} \rbrace)=D_j$ for $i,j=1,2$. Clearly we extend $\varphi'$ to $\varphi$ defining $\varphi(x_{j,1})=x_{j,2}$. It remains to show that $\varphi$ is order preserving and reflecting.\\
Let $p \in P_1$ be such that $p \notin \lbrace x_{1,1}, x_{2,1} \rbrace$.\\
Since $f_1,f_2$ are $\mathbf{P}$-morphisms and $f_i(x_{j,i})=g$ we get $x_{j,2}=\varphi(x_{j,1}) \leq \varphi(p)$ iff $g \leq f_2(\varphi(p))=f_1(p)$ iff $x_{j,1} \leq p$ for $j=1,2$.\\
Furthermore it is $p \leq x_{1,1}$ iff $f_1(p) \in D_1$ and $p \leq x_{2,1}$ iff $f_1(p) \in D_2$, similarly it is $\varphi(p) \leq x_{1,2}$ iff $f_1(p)=f_2(\varphi(p)) \in D_1$ and $\varphi(p) \leq x_{2,2}$ iff $f_1(p)=f_2(\varphi(p)) \in D_2$.\\
Therefore $\varphi$ is order preserving and reflecting.
\end{proof}

\begin{proof}[Proof of Theorem \ref{theosign}]
We just need to translate Lemma \ref{lemmasign} in the language of CBSes using the finite duality:\\
A signature of the first kind $(h,G)$ in $L_0$ corresponds to a couple $(D,U)$ in $P$ as in \ref{itm:lemmasign1} of Lemma \ref{lemmasign}. Indeed, by K\"{o}hler duality, downsets of $P$ correspond to elements of $L_0$ and upsets of $P$ correspond to the sets of their minimal elements, i.e. sets of two-by-two incomparable join-irreducible elements of $L_0$. The conditions $D \cap U = \emptyset$ and $\forall d \in D, u \in U$ $d \leq u$ translate in the condition $h < g$ for any $g \in G$.\\
A signature of the second kind $(h_1,h_2,g)$ in $L_0$ corresponds to a triple $(D_1,D_2,g)$ in $P$ as in \ref{itm:lemmasign2} of Lemma \ref{lemmasign}. Indeed, $h_1,h_2 \in L_0$ correspond to the downsets $D_1,D_2$ and $g$ join-irreducible of $L_0$ is an element of $P$ (recall that $P$ is the poset of the join-irreducibles of $L_0$). The condition that $D_1 \cup D_2 = \fregiu g \setminus \lbrace g \rbrace$ translates into $h_1 \vee h_2 = g^-$ since the predecessor $g^-$ of $g$ in $L_0$ corresponds to the downset $\fregiu g \setminus \lbrace g \rbrace$ of $P$.
\end{proof}

\textit{Therefore signatures inside a finite CBS $L_0$ are like `footprints' left by the minimal finite extensions of $L_0$: any minimal finite extension of $L_0$ leaves a `footprint' inside $L_0$ given by the corresponding signature. On the other hand, given a signature inside $L_0$ we can reconstruct a unique (up to isomorphism over $L_0$) minimal extension of $L_0$ corresponding to that signature.}\\

Since, by Theorems \ref{theoext1}, \ref{theoext2} and \ref{theominimalprimitive}, minimal finite extension of a finite CBS $L_0$ are exactly the ones generated over $L_0$ either by a primitive element or by a primitive couple, to any element or couple primitive over $L_0$ it is associated a unique signature in $L_0$. This is exactly what the next definition and theorem talk about.

\begin{definition}
Let $L_0$ be a finite CBS and $L$ an extension of $L_0$.\\
We say that a primitive element $x \in L$ of the first kind over $L_0$ induces a signature of the first kind $(h,G)$ in $L_0$ if for any $a$ join-irreducible of $L_0$ we have that
\[
a < x \text{ iff } a \leq h \quad \text{ and } \quad x < a \text{ iff } g \leq a \text{ for some } g \in G
\]
We say that a primitive couple $(x_1,x_2) \in L^2$ of the second kind over $L_0$ induces a signature of the second kind $(h_1,h_2,g)$ in $L_0$ if $g=x_1 \vee x_2$ and for any $a$ join-irreducible of $L_0$ we have that
\[
a < x_i \text{ iff } a \leq h_i \quad \text{ for } i=1,2
\]
\end{definition}

\begin{theorem} \label{theoinduce}
Let $L_0$ be a finite CBS and $L$ an extension of $L_0$.\\
A primitive element $x \in L$ induces a signature $(h,G)$ iff the extension $L_0 \subseteq L_0 \langle x \rangle$ corresponds to that signature.\\
A primitive couple $(x_1,x_2) \in L^2$ induces a signature $(h_1,h_2,g)$ iff the extension $L_0 \subseteq L_0 \langle x_1,x_2 \rangle$ corresponds to that signature.
\end{theorem}

\begin{proof}
For a primitive element $x$ of the first kind over $L_0$ to induce a signature $(h,G)$ means that $h$ is the predecessor of $x$ in $L_0 \langle x \rangle$ and $G$ is the set of the join-irreducibles of $L_0$ which are minimal among the ones that are strictly greater than $x$ in $L_0 \langle x \rangle$. This is the same as saying that the signature $(h,G)$ is associated to the extension $L_0 \subseteq L_0 \langle x \rangle$.\\
For a primitive couple $(x_1,x_2)$ of the second kind over $L_0$ to induce a signature $(h_1,h_2,g)$ means that $h_i$ is the predecessor of $x_i$ in $L_0 \langle x_1,x_2 \rangle$ for $i=1,2$. This is the same as saying that the signature $(h_1,h_2,g)$ is associated to the extension $L_0 \subseteq L_0 \langle x_1,x_2 \rangle$.
\end{proof}

We have thus finally obtained an intermediate characterization of existentially closed CBSes:

\begin{theorem}\label{theoaxiominterm}
A CBS $L$ is existentially closed iff for any finite sub-CBS $L_0 \subseteq L$ we have:
\begin{enumerate}
\item Any signature of the first kind in $L_0$ is induced by a primitive element $x \in L$ of the first kind over $L_0$.
\item Any signature of the second kind in $L_0$ is induced by a primitive couple $(x_1,x_2) \in L^2$ of the second kind over $L_0$.
\end{enumerate}
\end{theorem}

\begin{proof}
By the characterization of the existentially closed CBSes given in Theorem \ref{theoexclosedcharac} we have that a CBS $L$ is existentially closed iff for any finite sub-CBS $L_0$ and for any finite extension $L_0'$ of $L_0$ we have that $L_0'$ embeds into $L$ fixing $L_0$ pointwise. Since any finite extension of $L_0$ can be decomposed into a chain of minimal extensions, we can restrict to the case in which $L_0'$ is a minimal finite extension of $L_0$. Then the claim follows from Theorem \ref{theosign} and Theorem \ref{theoinduce}.
\end{proof}

\textit{Thanks to Theorem \ref{theoaxiominterm} we already get an axiomatization for the class of the existentially closed CBSes, indeed the quantification over the finite sub-CBS $L_0$ can be expressed elementarily using an infinite number of axioms. But this axiomatization is clearly unsatisfactory: other than being infinite, it is not conceptually clear.}

\section{Axioms} \label{section:axioms}

In this section we will prove that the existentially closed CBSes are exactly the ones satisfying the Splitting, Density 1 and Density 2 axioms. Each subsection focuses on one axiom. We will use extensively the characterization of existentially closed CBSes given by Theorem \ref{theoaxiominterm}. To show the validity of the axioms in any existentially closed CBS we will use the following lemma.

\begin{lemma} \label{lemdjcobrouw}
Let $\theta(\underline{x})$ and $\phi(\underline{x},\underline{y})$ be quantifier-free formulas in the language of CBSes. Assume that for every finite CBS $L_0$ and every tuple $\underline{a}$ of elements of $L_0$ such that $L_0 \vDash \theta(\underline{a})$, there exists an extension $L_1$ of $L_0$ which satisfies $\exists \underline{y} \phi(\underline{a},\underline{y})$.\\
Then every existentially closed CBS satisfies the following sentence:
\[ \forall \underline{x} ( \theta (\underline{x}) \longrightarrow \exists \underline{y} \phi(\underline{x},\underline{y})) \]
\end{lemma}

\begin{proof}
Let $L$ be an existentially closed CBS.\\
Let $\underline{a}=(a_1, \ldots,a_n) \in L^n$ be such that $L \vDash \theta(\underline{a})$. Let $L_0$ be the sub-CBS of $L$ generated by $a_1, \ldots, a_n$, by local finiteness $L_0$ is finite. By hypothesis there exists an extension $L_1$ of $L_0$ and $\underline{b}=(b_1, \ldots,b_m) \in L_1^m$ such that $L_1 \vDash \phi(\underline{a},\underline{b})$.\\
Denote by $L_0'$ the sub-CBS of $L_1$ generated by $b_1, \ldots,b_m$ over $L_0$, it is a finite extension of $L_0$. By Theorem \ref{theoexclosedcharac} $L_0'$ embeds into $L$ fixing $L_0$ pointwise.\\
We thus get $L \vDash \phi(\underline{a},\underline{b'})$ where $\underline{b'}=(b_1', \ldots,b_m') \in L^m$ are the images of $b_1, \ldots,b_m$ by the embedding.\\
Therefore we have proved that:
\[
L \vDash \forall \underline{x} ( \theta (\underline{x}) \longrightarrow \exists \underline{y} \phi(\underline{x},\underline{y}))
\]
\end{proof}

\subsection{Splitting axiom} 

\noindent\textbf{[Splitting Axiom]} For every $a,b_1,b_2$ such that $b_1 \vee b_2 \ll a \neq 0$ there exist elements $a_1$ and $a_2$ different from $0$ such that:
\begin{equation*}
\begin{split}
     a-a_1 &= a_2 \geq b_2 \\ 
     a-a_2 &= a_1 \geq b_1 \\
     b_2-a_1 &=b_2-b_1 \\
     b_1-a_2 &=b_1-b_2 
\end{split}
\end{equation*}

\begin{theorem}
Any existentially closed CBS satisfies the Splitting Axiom.
\end{theorem}

\begin{proof}
It is sufficient to show, by Lemma \ref{lemdjcobrouw}, that for any finite CBS $L_0$ and $a,b_1,b_2 \in L_0$ such that  $b_1 \vee b_2 \ll a \neq 0$ there exists a finite extension $L_0 \subseteq L$ with $a_1,a_2 \in L$ different from $0$ such that:
\begin{equation*}
\begin{split}
     a-a_1 &= a_2 \geq b_2 \\ 
     a-a_2 &= a_1 \geq b_1 \\
     b_2-a_1 &=b_2-b_1 \\
     b_1-a_2 &=b_1-b_2 
\end{split}
\end{equation*}
Let $Q$ be the poset dual to $L_0$ and $A, B_1,B_2$ its downsets corresponding to $a, b_1,b_2$.\\
We obtain a surjective $\mathbf{P}$-morphism $ \pi :P \to Q$ in the following way:\\
For any $x \in Q$ such that $x \notin B_2$ (respectively $x \notin B_1$) let $\xi_{x,1}$ (respectively $\xi_{x,2}$) be a new symbol.\\
For any $x \in Q$ such that $x \in B_1 \cap B_2$ let $\xi_{x,0}$ be a new symbol.\\
Let $P$ be the set of all these symbols, we define an order on $P$ setting:
\begin{equation*}
\begin{split}
	\xi_{y,j} \leq \xi_{x,i} & \Leftrightarrow y \leq x \text{ and } \lbrace i,j \rbrace \neq \lbrace 1,2 \rbrace
\end{split}
\end{equation*}
Intuitively $P$ is composed by a copy of $B_1 \cup B_2$ and two copies of $Q \backslash (B_1 \cup B_2)$, one of the two copies is placed over $B_1$ and the other over $B_2$.\\
We define $\pi :P \rightarrow Q$ setting $\text{dom } \pi =P$ and $\pi(\xi_{x,i})=x$.\\
Let $\fregiu a_1, \ldots, \fregiu a_r$ be the join-irreducible components of $A$, for any $i$ we have $a_i \notin B_1 \cup B_2$ because by hypothesis $B_1 \cup B_2 \ll A$. Therefore $\pi^{-1}(\fregiu a_i)= \fregiu \xi_{a_i,1} \cup \fregiu \xi_{a_i,2}$\\
We take:
\[
A_1= \bigcup\limits_{i=1}^r \fregiu \xi_{a_i,1} \qquad \text{ and } \qquad A_2= \bigcup\limits_{i=1}^r \fregiu \xi_{a_i,2}
\]
We obtain $\pi^{-1}(A)-A_1=A_2$ and $\pi^{-1}(A)-A_2=A_1$, they are both not empty because $r \geq 1$ and $A$ is not empty.\\
Furthermore for any $x \in B_1 \cup B_2$ we have that $x \leq a_i$ for some $i$. Therefore if $x \in B_1 \backslash B_2$ it is $\xi_{x,1} \leq \xi_{a_i,1}$, if $x \in B_2 \backslash B_1$ it is $\xi_{x,2} \leq \xi_{a_i,2}$, finally if $x \in B_1 \cap B_2$ then $\xi_{x,0} \leq \xi_{a_i,1}$ and $\xi_{x,0} \leq \xi_{a_i,2}$. This implies that $\pi^{-1}(B_1) \subseteq A_1$ and $\pi^{-1}(B_2) \subseteq A_2$.\\
We now show that $A_1 \cap A_2 = \pi^{-1} (B_1) \cap \pi^{-1}(B_2)$.\\
Let $\xi \in P$, we show that $\xi \in A_1 \cap A_2$ iff $\xi \in \pi^{-1} (B_1) \cap \pi^{-1}(B_2)$.\\
If $\xi \in \pi^{-1} (B_1) \cap \pi^{-1}(B_2)$ then $\pi(\xi) \in B_1 \cap B_2$, therefore $\xi=\xi_{x,0}$ and $x \leq a_i$ for some $i$. It implies that $\xi_{x,0} \leq \xi_{a_i,1}$, thus $\xi_{x,0} \in A_1$ and $\xi_{x,0} \leq \xi_{a_i,2}$, therefore $\xi_{x,0} \in A_2$ and $\xi \in A_1 \cap A_2$.\\
On the other hand if $\xi \in A_1 \cap A_2$ then there exist $i,j$ such that $\xi \leq \xi_{a_i,1}$ and $\xi \leq \xi_{a_j,2}$. By definition of the order on $P$ it has to be $\xi=\xi_{x,0}$ with $x \in B_1 \cap B_2$, therefore $\xi \in \pi^{-1} (B_1) \cap \pi^{-1}(B_2)$.\\
Then
\[
\pi^{-1} (B_1) \cap \pi^{-1}(B_2) \subseteq A_1 \cap \pi^{-1} (B_2) \subseteq A_1 \cap A_2 = \pi^{-1} (B_1) \cap \pi^{-1}(B_2)
\]
Therefore
\begin{equation*}
\begin{split}
	& \pi^{-1} (B_2)-A_1= \pi^{-1} (B_2)-(A_1 \cap \pi^{-1} (B_2)) \\
	& =\pi^{-1} (B_2)-(\pi^{-1} (B_1) \cap \pi^{-1} (B_2))=\pi^{-1} (B_2)-\pi^{-1} (B_1)
\end{split}
\end{equation*}
Analogously we can show
\[
 \pi^{-1} (B_2)-A_1=\pi^{-1} (B_2)-\pi^{-1} (B_1)
\]
Thus taking the embedding $L_0 \hookrightarrow L$ dual to $\pi$ and $a_1,a_2 \in L$ corresponding to $A_1,A_2$ we have obtained what we were looking for.
\end{proof}

\begin{lemma} \label{lemmaauxjoinirr}
If $L$ is a CBS generated by a finite subset $X$ then any join-irreducible element of $L$ is a join-irreducible component in $L$ of some element of $X$.
\end{lemma}

\begin{proof}
In any CBS the following identities hold:
\begin{equation*}\label{eq:plus} \tag{+}
\begin{split}
	c-(a \vee b) &= (c-a)-b \\
	(a \vee b)-c &= (a-c) \vee (b-c) \\
	c-0 &=c \\
	0-c &=0	
\end{split}
\end{equation*}
It follows by an easy induction that any term in the language of CBS is equivalent to a term of the form $x_1 \vee \cdots \vee x_n$ with $x_1, \ldots x_n$ containing only the difference symbol and variables. 
Notice that if an element $x_1 \vee \cdots \vee x_m$ with $x_1, \ldots x_m \in L$ is join-irreducible then it coincides with $x_i$ for some $i=1, \ldots,m$; thus any join-irreducible element $g$ of $L$ is the interpretation of a term $t$ over the variables $X$ containing only the difference symbol. This implies that $g$ is the join of some join-irreducible components of the leftmost variable in $t$. Indeed, this can be proved by induction on the complexity of the term observing that if $c_1, \ldots, c_m$ are the join-irreducible components of an element $c \in L$ then for any $b \in L$:
\[
c-b=(c_1 \vee \cdots \vee c_m)-b=(c_1-b)\vee \cdots \vee (c_m-b)= \bigvee_{c_i \nleq b} c_i
\]
because $c_i-b=0$ or $c_i-b=c_i$, respectively, when $c_i \leq b$ or $c_i \nleq b$ since the $c_i$'s are join-irreducibles.\\
Thus $g$ is the join of the join-irreducible components of some $x \in X$, so, since it is join-irreducible, it is a join-irreducible component of $x$. 
\end{proof}


\begin{remark*}
Lemma \ref{lemmaauxjoinirr} is not true for co-Heyting algebras.\\
Indeed, consider the inclusion $L_0 \hookrightarrow L_1$ of co-Heyting algebras described by Figure \ref{fig:lemmaextgenexample1}.
$L_1$ is generated by $L_0$ and $a$ but $b=a \wedge (1-a)$ is join-irreducible in $L_1$ and it is not a join-irreducible component of any element of $L_0$ or $a$.
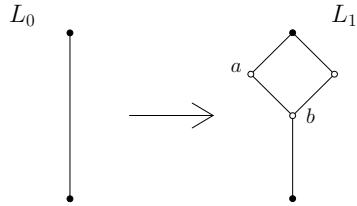
\begin{figure}[!h]
\centering
\scalebox{0.55}{
\begin{tikzpicture}[line cap=round,line join=round,>=triangle 45,x=1.0cm,y=1.0cm]
\clip(-5.,-0.5) rectangle (4.5,5.2);
\draw (-1.8988170326662965,2.00872402099691)-- (0.10118296733370835,2.00872402099691);
\draw (-0.3988170326662934,2.308724020996912)-- (0.10118296733370835,2.00872402099691);
\draw (-0.3988170326662934,1.7087240209969097)-- (0.10118296733370835,2.00872402099691);
\draw (1.0031829673337054,3.004724020996913)-- (2.003182967333706,4.004724020996913);
\draw (1.0031829673337054,3.004724020996913)-- (2.003182967333706,2.0047240209969104);
\draw (2.003182967333706,4.004724020996913)-- (3.003182967333706,3.004724020996913);
\draw (3.003182967333706,3.004724020996913)-- (2.003182967333706,2.0047240209969104);
\draw (2.003182967333706,2.0047240209969104)-- (2.003182967333706,0.004724020996902639);
\draw (-3.3215717243733187,4.000758124885548)-- (-3.3215717243733187,7.581248855379613E-4);
\draw (0.4,3.4) node[anchor=north west] {\Large $a$};
\draw (-4.9,4.8) node[anchor=north west] {\LARGE $L_0$};
\draw (2.8,4.8) node[anchor=north west] {\LARGE $L_1$};
\draw (2.2,2.3) node[anchor=north west] {\Large $b$};
\begin{scriptsize}
\draw [fill=ffffff] (1.0031829673337054,3.004724020996913) circle (2.0pt);
\draw [fill=black] (2.003182967333706,4.004724020996913) circle (2.0pt);
\draw [fill=ffffff] (2.003182967333706,2.0047240209969104) circle (2.0pt);
\draw [fill=ffffff] (3.003182967333706,3.004724020996913) circle (2.0pt);
\draw [fill=black] (2.003182967333706,0.004724020996902639) circle (2.0pt);
\draw [fill=black] (-3.3215717243733187,4.000758124885548) circle (2.0pt);
\draw [fill=black] (-3.3215717243733187,7.581248855379613E-4) circle (2.0pt);
\end{scriptsize}
\end{tikzpicture}
}
\caption{The inclusion $L_0 \hookrightarrow L_1$} \label{fig:lemmaextgenexample1}
\end{figure}
\end{remark*}

\begin{lemma} \label{lemmaauxdual}
Let $L_0$ be a finite sub-CBS of $L$ and let $L$ be generated by $L_0$ and $a_1, \ldots, a_n \in L$.\\
If $a_1, \ldots, a_n$ are joins of join-irreducible components in $L$ of elements of $L_0$\footnotemark then the surjective $\mathbf{P}$-morphism $\varphi: P \to Q$ dual to the inclusion $L_0 \hookrightarrow L$ is such that $\text{dom } \varphi =P$. In particular, the inclusion is also a co-Heyting algebras morphism, i.e. it preserves meets and $1$.
\footnotetext{For instance, this happens when we have that all the $a_i$'s are of the kind $a_i=b_i-c_i$ with $b_i \in L_0$ and $c_i \in L$. Indeed, in such case, the join-irreducible components in $L$ of the $a_i$'s are among those of $b_i$.}
\end{lemma}

\begin{proof}
By Lemma \ref{lemmaauxjoinirr} all the join-irreducible elements of $L$ are join-irreducible components in $L$ of elements of $L_0$ or of $a_1, \ldots, a_n$. Since, by hypothesis, $a_1, \ldots, a_n$ are joins of join-irreducible components of elements of $L_0$, any join-irreducible components of $a_i$ is a join-irreducible component of a join-irreducible component of an element of $L_0$ and thus it is a join-irreducible component of an element of $L_0$. Therefore any join-irreducible element of $L$ is a join-irreducible component in $L$ of an element of $L_0$.\\
Suppose that there is $x \in P$ such that $x \notin \text{dom } \varphi$, then $x$ corresponds to a join-irreducible element of $L$ which is not a join-irreducible component of any element of $L_0$. Indeed, if it is a join-irreducible component of $a \in L_0$ then $x \in P$ would be a maximal element of the downset $\fregiu \varphi^{-1} (A)$ where $A \subseteq Q$ is the downset relative to $a$, but this is not possible since if $x \in \fregiu \varphi^{-1} (A)$ then $x$ would be less than or equal to an element in $\varphi^{-1} (A) \subseteq \text{dom } \varphi$ which would be different from $x$, this is absurd because $x$ is maximal in $\fregiu \varphi^{-1} (A)$. Therefore, $\text{dom } \varphi=P$ because the existence of an element $x \notin \text{dom } \varphi$ would imply the existence of a join-irreducible element of $L$ which is not a join-irreducible component in $L$ of any element of $L_0$, but this contradicts what we have proven in the first part of this proof.
\end{proof}

\begin{lemma} \label{lemmaauxsplit}
Let $L$ be a CBS and $L_0$ a finite sub-CBS of $L$, $g$ be join-irreducible in $L_0$ and $y_1,y_2 \in L$ be nonzero elements such that 
\begin{equation*}
\begin{split}
     g-y_1 &= y_2 \\ 
     g-y_2 &= y_1
\end{split}
\end{equation*}
Let also $L_0 \langle y_1,y_2 \rangle$ be the sub-CBS of $L$ generated by $L_0$ and $\lbrace y_1,y_2 \rbrace$. We have that:
\begin{enumerate}
\item $g=y_1 \vee y_2$,
\item any join-irreducible $a$ of $L_0$ such that $a \nleq g$ is still join-irreducible in $L_0 \langle y_1,y_2 \rangle$,
\item $y_1,y_2$ are distinct, not in $L_0$ and they are the join-irreducible components of $g$ in $L_0 \langle y_1,y_2 \rangle$.
\end{enumerate}
\end{lemma}

\begin{proof}
Notice that $y_1 \vee y_2=g$ because $y_1 \leq g$ and $y_2 \leq g$ and 
\[
g-(y_1 \vee y_2)=(g-y_1)-y_2=y_2-y_2=0.
\]
Furthermore $y_1,y_2 \notin L_0$. 
Indeed, suppose that $y_1 \in L_0$, then $y_2=g-y_1 \in L_0$, since $g$ is join-irreducible in $L_0$ and $g=y_1 \vee y_2$, we have that $g=y_1$ or $g=y_2$, by hypothesis it follows respectively that $y_2=0$ or $y_1=0$, in both cases we have a contradiction because $y_1,y_2 \neq 0$. Similarly, we obtain that $y_2 \notin L_0$.\\
We also have that $y_1 \neq y_2$. Indeed, suppose $y_1=y_2$, then $g-y_1=y_1$ implies that $g=y_1=0$ and this is absurd.\\ 
We now show that any \textit{join-irreducible $a$ of $L_0$ such that $a \nleq g$ is still join-irreducible in $L_0 \langle y_1,y_2 \rangle$}.\\
Any element of $L_0 \langle y_1,y_2 \rangle$ is the join of repeated differences of $y_1,y_2$ and of join-irreducibles of  $L_0$ different from $g$, this is implied by the identities \eqref{eq:plus} as noted in the proof of Lemma \ref{lemmaauxjoinirr}.\\
It is sufficient to show that for any $x$ obtained as repeated differences of $y_1,y_2$ and join-irreducibles of  $L_0$ different from $g$ we have $a-x=a$ or $a-x=0$. This will ensures that $a$ is join-irreducible in $L \langle y_1,y_2 \rangle$. 
Since $x$ is obtained as repeated differences of $y_1,y_2$ and join-irreducibles $a_1, \ldots, a_n$ of $L_0$ different from $g$, there is a term $t$ in the language of CBSes containing only $-$ and variables expressing $x$ as $t(y_1,y_2,a_1, \ldots,a_n)$. We prove that $a-x=a$ or $a-x=0$ by induction on the length of $t$. If the length is $1$, then $x \in \lbrace y_1, y_2, a_1, \ldots, a_n \rbrace$ \\
If $x \in L_0$ then $a-x=a$ or $a-x=0$ since $a$ is join-irreducible of $L_0$. Moreover $a-y_i=a$ for $i=1,2$. Indeed, $a \geq a-y_i \geq a-g=a$ because $a$ is a join-irreducible of $L_0$ such that $a \nleq g$. \\
Suppose the length of $t$ is greater than $1$. If the leftmost element among the ones whose differences give $x$ is $y_i$ for $i=1,2$ then $a-x=a$ because $a=a-y_i \leq a-x \leq a$.\footnote{Notice the following fact: if a term $t$ in the language of CBSes only contains $-$ (and variables) and $z$ is the leftmost variable in $t$, then the inequality $t \leq z$ is valid in every CBS (this is established by an easy induction on the length of $t$).\label{fnlengthterm}}\\
If the leftmost element among the ones whose differences give $x$ is $b$ join-irreducible of $L_0$ different from $g$. If $b < g$, since $a \nleq g$ yields $a \nleq b$ and thus $a=a-b$ by join-irriducibility of $a$, we obtain that $a=a-b \leq a-x \leq a$. If $b \nleq g$ 
we can obtain $x$ with a smaller number of differences because we can apply the induction hypothesis to replace its subterm of the kind $b-c$ with $b$ or $0$ (with $b$ and $c$ playing respectively the role previously played by $a$ and $x$) and apply again the inductive hypothesis because $x$ can be expressed by a term shorter than $t$.\\ 
Now we prove that \textit{$y_1,y_2$ are join-irreducibles of $L_0 \langle y_1,y_2 \rangle$.} We show that for $y_1$, for $y_2$ is analogous.\\
Again, we show that for all $x \in L_0 \langle y_1,y_2 \rangle$ we have $y_1-x=y_1$ or $y_1-x=0$.
Let $x$ be obtained as repeated differences of $y_1,y_2$ and join-irreducibles of  $L_0$ different from $g$ as above and let us proceed by induction on the number of such differences.\\
First of all
\[
y_1-y_2=(g-y_2)-y_2=(g-y_2)=y_1.
\]
Let $x \in L_0$, if $g \leq x$ then $y_1-x \leq y_1-g=0$, if $g \neq x$, namely $g-x=g$ (recall that $g$ is join-irreducible in $L_0$) then
\[
y_1-x=(g-y_2)-x=(g-x)-y_2=g-y_2=y_1.
\]
If the leftmost element among the ones whose differences give $x$ is $y_i$ for $i=1,2$ then applying the inductive hypothesis, possibly many times, we obtain that $x=0$ or $x=y_i$ and in either case $y_1-x=0$ or $y_1-x=y_1$.\\
Suppose the leftmost element among the ones whose differences give $x$ is $b$ join-irreducible of $L_0$ different from $g$. If $b < g$ then 
\[
y_1=(g-y_2)=(g-b)-y_2=y_1-b \leq y_1
\]
and $y_1=y_1-b \leq y_1-x \leq y_1$.\footnote{Recall footnote \ref{fnlengthterm}.}\\
If $b \nleq g$ we cannot have $b=g$ because $g$ is not join-irreducible, then using what we have proved above (that any join-irreducible $b$ of $L_0$ such that $b \nleq g$ is still join-irreducible in $L_0 \langle y_1,y_2 \rangle$) we obtain $x=b$ or $x=0$ and in either case $y_1-x=0$ or $y_1-x=y_1$.\\
Finally, to prove that \textit{$y_1,y_2$ are the join-irreducible components of $g$ in $L_0 \langle y_1,y_2 \rangle$}, we simply have to notice that $y_1 \nleq y_2$ and $y_2 \nleq y_1$. Just observe that if $y_1 \leq y_2$ then $g=y_1 \vee y_2 = y_2 \notin L_0$ which is absurd. Analogously it cannot be $y_2 \leq y_1$.
\end{proof}

\begin{theorem} \label{theosecondextaxiom}
Let $L$ be a CBS satisfying the Splitting Axiom.\\
Then for any finite sub-CBS $L_0 \subseteq L$ and for any signature $(h_1,h_2,g)$ of the second kind in $L_0$ there exists a primitive couple $(x_1,x_2) \in L^2$ of the second kind over $L_0$ inducing such signature.
\end{theorem}

\begin{proof}
We follow this strategy: we use the Splitting Axiom to `split' $g$ obtaining two elements, one over $h_1$ and another over $h_2$. When $h_1=h_2$ these two elements form a primitive couple that induces the signature $(h_1,h_2,g)$, but unfortunately this is not true in general because these two elements may be too big. So we may have to `split' these two new elements too in order to obtain other elements which we may have to `split' again and again. This process has to stop after a finite number of steps, intuitively the more the element $h_1 \wedge h_2$ (the meet is taken inside $L_0$) is smaller that $h_1$ and $h_2$, the more the process lasts. Then we accurately partition the set of all these `shards' into two disjoint subsets and we take the joins of these two subsets. In this way we obtain two elements that form a primitive couple that induces the signature $(h_1,h_2,g)$ and we are done.\\
The statement of the Theorem require that, according to the definition of primitive couple inducing a given signature, we need to do the following. Given $h_1,h_2 \in L_0$ and $g$ join-irreducible of $L_0$ such that $h_1 \vee h_2 = g^-$, we have to find $x_1, x_2 \in L$ such that:
\begin{enumerate}
\item $x_1 \neq x_2$ and $x_1,x_2 \notin L_0$, \label{itm:secextsplaxcond1}
\item $g-x_1=x_2$ and $ g-x_2=x_1$ \label{itm:secextsplaxcond2}
\end{enumerate}
and for any $a$ join-irreducible of $L_0$:
\begin{enumerate}[resume]
\item if $a <g$ then $a-x_i \in L_0$ for $i=1,2$, \label{itm:secextsplaxcond3}
\item $a < x_i$ iff $a \leq h_i \;$ for $i=1,2$. \label{itm:secextsplaxcond4}
\end{enumerate}
We recall that $L_0$ is a co-Heyting algebra because it is finite. In particular we can consider meets inside $L_0$ and they distribute with the joins.\\
Let $n_i$ for $i=1,2$ the maximum length of chains of join-irreducible elements of $L_0$
\[
k_1 < k_2 < \cdots <k_{n_i}
\]
such that $k_{n_i} \leq h_i$ and $k_1 \nleq h_j$ with $i \neq j$, or equivalently $k_1 \nleq h_1 \wedge h_2$ where the meet is taken inside $L_0$.\\
Let $n=n_1+n_2$.\\
Intuitively, the natural number $n$ measures how much $h_1 \wedge h_2$ is smaller than $h_1$ and $h_2$.\\
We prove the claim by induction on $n$.\\
\framebox{Case 1: $n=0$.}\\
Then $h_1 \wedge h_2=h_1=h_2=g^-$. We denote $h_1=h_2$ by $h$.\\
Since $h \ll g$, we can apply the splitting axiom to $g,h,h$, hence there exist elements $x_1,x_2 \in L$ different from $0$ such that:
\begin{equation} \label{eq:casen=0splitresult}
\begin{split}
     g-x_1 &= x_2 \geq h \\ 
     g-x_2 &= x_1 \geq h 
\end{split}
\end{equation}
We now show that $(x_1,x_2)$ is a primitive couple of the second kind and induces the signature $(h_1,h_2,g)$:
\begin{enumerate}
\item 
	As shown in Lemma \ref{lemmaauxsplit} we have that $x_1 \neq x_2$ and $x_1,x_2 \notin L_0$.
\item 
	$g-x_1=x_2$ and $ g-x_2=x_1$ follow directly from the splitting axiom, see \eqref{eq:casen=0splitresult}.
\end{enumerate}
Let $a$ be a join-irreducible element of $L_0$, then for $i=1,2$:
\begin{enumerate}[resume]
\item 
	If $a < g$ then $a \leq g^-= h$ thus $a-x_i=0$ because $h \leq x_i$ as a consequence of the splitting axiom, see \eqref{eq:casen=0splitresult}.
\item 
	If $a < x_i$ then $a < g$ because $x_i < g$ by \eqref{eq:casen=0splitresult} and therefore $a \leq g^- =h$.\\
	If $a \leq h$ then, since $x_i \notin L_0$ and $h \leq x_i$, we have $a < x_i$.  
\end{enumerate}
\framebox{Case 2: $n>0$.}\\
Suppose that the claim is true for any $m<n$.\\
Since $h_1 \vee h_2 =g^- \ll g$, we can apply the splitting axiom to $g,h_1,h_2$, hence there exist elements $y_1,y_2 \in L$ different from $0$ such that:
\begin{equation} \label{eq:casen>0splitresult}
\begin{split}
     g-y_1 &= y_2 \geq h_2 \\ 
     g-y_2 &= y_1 \geq h_1 \\
	 h_2-y_1 &=h_2-h_1 \\
     h_1-y_2 &=h_1-h_2 
\end{split}
\end{equation}
Let $L_0 \langle y_1,y_2 \rangle$ be the sub-CBS of $L$ generated by $L_0$ and $\lbrace y_1,y_2 \rbrace$. By local finiteness $L_0 \langle y_1,y_2 \rangle$ is finite and thus a co-Heyting algebra.\\
Before continuing the proof we show a series of claims.
\begin{claim}
The two following triples are signatures of the second kind in $L_0 \langle y_1,y_2 \rangle$:
\begin{equation} \label{triplesinductsplit}
\begin{split}
	(h_1, \; h_2 \wedge y_1, \; y_1)\\
	(h_1 \wedge y_2, \; h_2, \; y_2)
\end{split}
\end{equation}
where the meets are taken inside $L_0 \langle y_1,y_2 \rangle$.
\end{claim}
\begin{claimproof}
By Lemma \ref{lemmaauxsplit} $y_1,y_2 \notin L_0$ are join-irreducibles in $L_0 \langle y_1,y_2 \rangle$.\\
Moreover, in $L_0 \langle y_1,y_2 \rangle$ we have that:
\begin{equation} \label{triplespredecessor}
\begin{split}
	h_1 \vee (h_2 \wedge y_1) &=y_1^- \\
	(h_1 \wedge y_2) \vee h_2 &=y_2^-
\end{split}
\end{equation}
Indeed
\[
h_1 \vee (h_2 \wedge y_1)=(h_1 \vee h_2) \wedge (h_1 \vee y_1)=(h_1 \vee h_2) \wedge y_1
\]
and we have that this coincides with $y_1^-$, the predecessor of $y_1$ in $L_0 \langle y_1,y_2 \rangle$. To show this, observe that, as a consequence of Lemma \ref{lemmaauxdual}, the inclusion $L_0 \hookrightarrow L_0 \langle y_1,y_2 \rangle$ is dual to a surjective $\mathbf{P}$-morphism $\varphi:P \to Q$ with $\text{dom } \varphi= P$. Recall that $Q$ and $P$ are the posets of the join-irreducibles respectively of $L_0$ and $L_0 \langle y_1,y_2 \rangle$. Notice that the preimage of an element $q$ of $Q$, i.e. a join-irreducible element of $L_0$ consists of the join-irreducible components of such element inside $L_0 \langle y_1,y_2 \rangle$. This is a consequence of the fact that $\fregiu \varphi^{-1}(q) \subseteq P$ corresponds to $q$ as element of $L_0 \langle y_1,y_2 \rangle$ and the set of maximal elements of  $\fregiu \varphi^{-1}(q)$ is exactly the preimage of $q$ because $\varphi$ preserves the strict order.\\
Then $\varphi^{-1}(g)= \lbrace y_1, y_2 \rbrace$ because $y_1,y_2$ are the join-irreducible components  of $g$ in $L_0 \langle y_1,y_2 \rangle$ by Lemma \ref{lemmaauxsplit}.\\
Notice that for any downset $D \subseteq Q$ we have $\fregiu \varphi^{-1}(D)=\varphi^{-1}(D)$ because $\text{dom } \varphi= P$ and  $\varphi^{-1}(D)$ is a downset in the domain of $\varphi$.\\
Since $\fregiu \varphi^{-1}(\fregiu g)=\varphi^{-1}(\fregiu g)=\fregiu y_1 \cup \fregiu y_2$ we have:
\[
\fregiu \varphi^{-1}(\fregiu g \setminus \lbrace g \rbrace)=\varphi^{-1}(\fregiu g \setminus \lbrace g \rbrace)=(\fregiu y_1 \cup \fregiu y_2) \setminus \lbrace y_1,y_2 \rbrace
\]
Thus $(h_1 \vee h_2)\wedge y_1=g^- \wedge y_1$ is equal to $y_1^-$ because of the following equations
\[
(\fregiu y_1 \cup \fregiu y_2) \setminus \lbrace y_1,y_2 \rbrace \cap \fregiu y_1 = \fregiu y_1 \setminus  \lbrace y_1,y_2 \rbrace= \fregiu y_1 \setminus  \lbrace y_1 \rbrace
\]
To prove that $(h_1 \wedge y_2) \vee h_2 =y_2^-$ the reasoning is analogous.
\end{claimproof}
Notice that $h_1 \wedge h_2$ is the same taken in $L_0$ and in $L_0 \langle y_1,y_2 \rangle$ because by Lemma \ref{lemmaauxdual} the inclusion $L_0 \hookrightarrow L_0 \langle y_1,y_2 \rangle$ preserves meets.
\begin{claim} 
For $i=1,2$ the maximum length of chains of join-irreducibles of $L_0 \langle y_1,y_2 \rangle$ less than or equal to $h_i$ that are not less than or equal to $h_1 \wedge h_2$ is the same as $n_i$ defined above (recall that $n_i$ is defined taking the join-irreducibles of $L_0$).
\end{claim}
\begin{claimproof}
Suppose there exists a chain of join-irreducibles in $L_0 \langle y_1,y_2 \rangle$
\[
k_1 < k_2 < \cdots <k_r
\]
such that $k_r \leq h_i$ and $k_1 \nleq h_1 \wedge h_2$. Let, as above, $\varphi:P \to Q$ the surjective total $\mathbf{P}$-morphism dual to the inclusion $L_0 \hookrightarrow L_0 \langle y_1,y_2 \rangle$, then
\[
\varphi(k_1) < \varphi(k_2) < \cdots < \varphi(k_r)
\]
is a chain of join-irreducibles in $L_0$ such that $\varphi(k_r) \leq h_i$ and $\varphi(k_1) \nleq h_1 \wedge h_2$.
\footnote{We remind that $\mathbf{P}$-morphisms preserve the strict order.}\\
On the other hand a chain of join-irreducibles in $L_0$
\[
b_1 < b_2 < \cdots <b_r
\]
such that $b_r \leq h_i$ and $b_1 \nleq h_1 \wedge h_2$ can be lifted to a chain of join-irreducibles of $L_0 \langle y_1,y_2 \rangle$
\[
k_1 < k_2 < \cdots <k_r
\]
such that $\varphi(k_s)=b_s$ for $s=1, \ldots, r$ using the fact that $\varphi$ is a surjective $\mathbf{P}$-morphism, we obtain that $k_r \leq h_i$ and $k_1 \nleq h_1 \wedge h_2$.
\end{claimproof}
\begin{claim}
If $h_2 \nleq h_1$ then the maximum length of chains of join-irreducibles of $L_0 \langle y_1,y_2 \rangle$ less than or equal to $h_2 \wedge y_1$ that are not less than or equal to $h_1 \wedge h_2= h_1 \wedge (h_2 \wedge y_1)$ is strictly smaller than $n_2$ (notice that $n_2 \neq 0$ because $h_2 \nleq h_1$).\\
When $h_1 \nleq h_2$ in the same way we obtain the analogous result switching $y_1$ with $y_2$ and $h_1$ with $h_2$.
\end{claim}
\begin{claimproof}
Suppose there exists a chain of join-irreducibles in $L_0 \langle y_1,y_2 \rangle$
\[
k_1 < k_2 < \cdots <k_{n_2}
\]
such that $k_{n_2} \leq h_2 \wedge y_1$ and $k_1 \nleq h_1 \wedge h_2$. Notice that this chain is not empty because $n_2 \neq 0$. We have that $k_{n_2}$ is not a join-irreducible component of $h_2$ in $L_0 \langle y_1,y_2 \rangle$. Indeed, $k_{n_2} \leq y_1$ and $y_1$ is not greater than or equal to any join-irreducible component of $h_2$ which is not less than or equal to $h_1 \wedge h_2$ because $h_2 - y_1=h_2-h_1=h_2-(h_1 \wedge h_2)$. Thus there would exist a continuation of such chain given by $k_{n_2+1}$ join-irreducible component of $h_2$ in $L_0 \langle y_1,y_2 \rangle$, but this is absurd because we have proved in Claim 2 that $n_2$ is the maximum length of such chains.
\end{claimproof}
We can now apply the inductive hypothesis, to do so we shall consider different cases.\\
\framebox{Subcase 2.1: $h_1,h_2$ incomparables.} First, we consider the case in which $h_1 \nleq h_2$ and $h_2 \nleq h_1$, i.e. $h_1,h_2$ are incomparable.\\
What we have proved in Claim 3 implies that the sum of the lengths of the chains considered above for either of the two signatures \eqref{triplesinductsplit} is strictly smaller than $n$. Therefore we can apply the inductive hypothesis on both the two signatures \eqref{triplesinductsplit} considered inside $L_0 \langle y_1,y_2 \rangle$ to obtain two primitive couples $(y_{11}, y_{12}) \in L^2$ and $(y_{21}, y_{22}) \in L^2$ of the second kind over $L_0 \langle y_1,y_2 \rangle$ such that they induce respectively the signatures $(h_1, h_2 \wedge y_1, y_1)$ and $(h_1 \wedge y_2,  h_2, y_2)$. This means that:
\begin{enumerate}
\item
	$y_{11} \neq y_{12}$ and $y_{11},y_{12} \notin L_0 \langle y_1,y_2 \rangle$, \label{itm:y1icond1}
\item 
	$y_1- y_{11}=y_{12}$ and $y_1- y_{12}=y_{11}$ \label{itm:y1icond2}
\end{enumerate}
and for any $a$ join-irreducible of $L_0 \langle y_1,y_2 \rangle$:
\begin{enumerate}[resume]
\item 
	if $a<y_1$ then $a-y_{1i} \in L_0 \langle y_1,y_2 \rangle \;$ for $i=1,2$, \label{itm:y1icond3}
\item 
	$a < y_{11}$ iff $a \leq h_1 \;$ and $\; a < y_{12}$ iff $a \leq (h_2 \wedge y_1)$. \label{itm:y1icond4}
\end{enumerate}
furthermore
\begin{enumerate}
\item
	 $y_{21} \neq y_{22}$ and $y_{21},y_{22} \notin L_0 \langle y_1,y_2 \rangle$, \label{itm:y2icond1}
\item 
	$y_2- y_{21}=y_{22}$ and $y_2- y_{22}=y_{21}$ \label{itm:y2icond2}
\end{enumerate}
and for any $a$ join-irreducible of $L_0 \langle y_1,y_2 \rangle$:
\begin{enumerate}[resume]
\item 
	if $a<y_2$ then $a-y_{2i} \in L_0 \langle y_1,y_2 \rangle \;$ for $i=1,2$, \label{itm:y2icond3}
\item 
	$a < y_{21}$ iff $a \leq (h_1 \wedge y_2) \;$ and $\; a < y_{22}$ iff $a \leq h_2$. \label{itm:y2icond4}
\end{enumerate}
Notice that properties \ref{itm:y1icond4} of $y_{11}, y_{12}$ and \ref{itm:y2icond4} of $y_{21}, y_{22}$ actually hold for any $a \in L_0 \langle y_1,y_2 \rangle$ since any element in a finite CBS is the join of join-irreducible elements. Observe also that for $a \in L_0$ we have $ a \leq y_{ij}$ iff $a < y_{ij}$ because $y_{ij} \notin L_0$.\\
We want to prove that $x_1=y_{11} \vee y_{21}$ and $x_2=y_{12} \vee y_{22}$ are the two elements of $L$ we are looking for, i.e. $(x_1,x_2)$ is a primitive couple of the second kind over $L_0$ inducing the signature $(h_1,h_2,g)$.\\
First of all, we observe that
\begin{equation} \label{eq:yiminusyjkfirstcase}
y_1-y_{2i}=y_1 \quad \text{ and } \quad y_2-y_{1i}=y_2 \qquad \text{for } i=1,2
\end{equation}
Indeed,
\begin{equation} \label{eq:yiminusyjkfirstcaseproof}
y_1=g-y_2=(g-y_2)-y_2=y_1-y_2 \leq y_1-y_{2i} \leq y_1
\end{equation}
the second equation is shown analogously. \eqref{eq:yiminusyjkfirstcaseproof} also shows that $y_1-y_2=y_1$ and $y_2-y_1=y_2$.\\
Moreover
\begin{equation} \label{eq:yihminusyjk}
y_{1i}-y_{2j}=y_{1i} \quad \text{ and } \quad y_{2i}-y_{1j}=y_{2i} \qquad \text{for } i,j=1,2
\end{equation}
Indeed, 
\[
y_{11}=y_1-y_{12}=(y_1-y_2)-y_{12}=(y_1-y_{12})-y_2=y_{11}-y_2 \leq y_{11} - y_{21} \leq y_{11}
\]
and thus $y_{11} - y_{21}=y_{11}$, the remaining cases are analogous.\\
Notice the following fact about the extensions generated by the $y_{ij}$'s:
\begin{claim}\label{remarkdoubleextyij}
The two extensions of finite CBSes given by $L_0 \langle y_1,y_2 \rangle \subseteq L_0 \langle y_1,y_2,y_{11},y_{12} \rangle \subseteq L_0 \langle y_{ij} \: | \: i,j=1,2 \rangle$ are both minimal of the second kind. This implies that any $b$ join-irreducible of $L_0 \langle y_1,y_2 \rangle$ different from $y_1,y_2$ is still join-irreducible in $L_0 \langle y_{ij} \: | \: i,j=1,2 \rangle$.
\end{claim}
\begin{claimproof}
It suffices to prove that $(y_{21}, y_{22})$ is a primitive couple of the second kind over $L_0 \langle y_1,y_2, y_{11}, y_{12} \rangle$.\\
First of all, as a consequence of Lemma \ref{lemmaauxsplit}, $y_1,y_2$ are join-irreducible in $L_0 \langle y_1,y_2 \rangle$, thus $y_2$ is join-irreducible in $L_0 \langle y_1,y_2, y_{11}, y_{12} \rangle$.
\begin{enumerate}
\item 
	$y_{21} \neq y_{22}$ by property \ref{itm:y2icond1} of $y_{21}, y_{22}$.\\
	$y_{21},y_{22} \in L \setminus L_0 \langle y_1,y_2, y_{11}, y_{12} \rangle$. Indeed, if $y_{21} \in L_0 \langle y_1,y_2, y_{11}, y_{12} \rangle$ then $y_{22}=y_2-y_{21} \in L_0 \langle y_1,y_2, y_{11}, y_{12} \rangle$ and vice versa. In that case $y_2=y_{21} \vee y_{22} \in L_0 \langle y_1,y_2, y_{11}, y_{12} \rangle$ with $y_{21}, y_{22} \neq y_2$ because they are not in $L_0 \langle y_1,y_2 \rangle$, but this is absurd because $y_2$ is join-irreducible.
\item
	$y_2-y_{21}=y_{22}$ and $y_2-y_{22}=y_{21}$ by property \ref{itm:y2icond2} of $y_{21}, y_{22}$.
\item
	Since $L_0 \langle y_1,y_2, y_{11}, y_{12} \rangle$ is a minimal finite extension of $L_0 \langle y_1,y_2 \rangle$, its join-irreducibles are $y_{11}, y_{12}$ and the join-irreducibles of $L_0 \langle y_1,y_2 \rangle$ except $y_1$. If $a$ is a join-irreducible of $L_0 \langle y_1,y_2, y_{11}, y_{12} \rangle$ such that $a < y_2$ then $a$ is join-irreducible in $L_0 \langle y_1,y_2 \rangle$ because $a \neq y_{11}, y_{12}$ since $y_{11}, y_{12} \nless y_2$: indeed $y_{1i}-y_2=y_{1i}-(y_{21} \vee y_{22})=(y_{1i}-y_{21})-y_{22}=y_{1i} \neq 0$ by \eqref{eq:yihminusyjk}.  Thus $a-y_{2i} \in L_0 \langle y_1,y_2 \rangle$ by property \ref{itm:y2icond3} of of $y_{21}, y_{22}$.
\end{enumerate}
\end{claimproof}
Moreover, we observe that
\begin{equation} \label{eq:gminusxifirstcase}
g-x_1=x_2 \quad \text{ and } \quad g-x_2=x_1
\end{equation}
because thanks to equations \eqref{eq:yiminusyjkfirstcase} we have:
\begin{equation*} 
	\begin{split}
		g-x_1=(y_1 \vee y_2)-(y_{11} \vee y_{21})=((y_1-y_{21})-y_{11}) \vee ((y_2 - y_{11})-y_{21})\\
		=(y_1-y_{11}) \vee (y_2-y_{21})=y_{12} \vee y_{22}=x_2;
	\end{split}
	\end{equation*}
showing the second equation of \eqref{eq:gminusxifirstcase} is analogous.\\
We are now ready to show that $(x_1,x_2)$ is a primitive couple of the second kind over $L_0$ inducing the signature $(h_1,h_2,g)$.
\begin{enumerate}
\item 
	Equations \eqref{eq:gminusxifirstcase} imply that $g=x_1 \vee x_2$, thus if $x_1=x_2$ then $x_1=g=0$ but this is absurd because $x_1,x_2 \neq 0$ since $y_{11}, y_{12}, y_{21}, y_{22} \neq 0$ because they are not in $L_0$.\\
Furthermore $x_1,x_2 \notin L_0$; this is because $g$ is join-irreducible in $L_0$ and $g-x_1=x_2$ and $g-x_2=x_1$ are different from $0$ and $g$.
\item
	See equations \eqref{eq:gminusxifirstcase}.
\end{enumerate}
Let now $a$ be a join-irreducible element of $L_0$ and $i=1,2$:
\begin{enumerate}[resume]
\item 
	If $a < g$ then $a \leq g^-=h_1 \vee h_2$ and $a \leq h_1$ or $a \leq h_2$
	\begin{itemize}
	\item
		If $a \leq h_1$ then $a-x_1=0$ because $h_1 \leq y_{11} \leq x_1$ by property \ref{itm:y1icond4} of $y_{11}$.
	\item
		If $a \leq h_2$ and $a \nleq h_1$ we want to prove that $a-x_1=a$.
		\begin{claim}
		The join-irreducible components of $a$ in $L_0 \langle y_1,y_2 \rangle$ coincide with the join-irreducible component of $a$ in $L_0 \langle y_{ij} \: | \: i,j=1,2 \rangle$.
		\end{claim}
		\begin{claimproof}
		Since $a$ is the join of its join-irreducible components in $L_0 \langle y_1,y_2 \rangle$, it is sufficient to prove that any join-irreducible component $b$ of $a$ in $L_0 \langle y_1,y_2 \rangle$ is join-irreducible in $L_0 \langle y_{ij} \: | \: i,j=1,2 \rangle$. We have $b \neq y_1,y_2$ because $b \leq h_2$ and $y_1,y_2 \nleq h_2$ since if $y_i \leq h_2$ then $0=y_i -h_2=(g-y_j)-h_2=(g-h_2)-y_j=g-y_j=y_i$  with $i \neq j$ which is absurd. Thus by Claim 4 we have that $b$ is also join-irreducible in $L_0 \langle y_{ij} \: | \: i,j=1,2 \rangle$.
		\end{claimproof}
		Since $a$ is join-irreducible of $L_0$ and $a \nleq h_1$ it is $a-h_1=a$. For any $b$ join-irreducible component of $a$ in $L_0 \langle y_1,y_2 \rangle$ we have $b \nleq h_1$ because $a-h_1=a$ means that $h_1$ is not greater than or equal to any join-irreducible component of $a$. Since $b \nleq h_1$ and in particular $b \nleq h_1 \wedge y_2$ then property \ref{itm:y1icond4} of $y_{11}$ and property \ref{itm:y2icond4} of $y_{21}$ imply that $b \nleq y_{11}, y_{21}$. Therefore $b \nleq y_{11} \vee y_{21}=x_1$ because $b$ is join-irreducible in $L_0 \langle y_{ij} \: | \: i,j=1,2 \rangle$. This implies that $a-x_1=a$ because $x_1$ is not greater than or equal to any join-irreducible component of $a$ in $L_0 \langle y_{ij} \: | \: i,j=1,2 \rangle$.\\
	For $a-x_2$ the property is checked in an analogous way.
	\end{itemize}
\item 
	If $a \leq h_i$ then $a < y_{ii} \leq x_i$ by property \ref{itm:y1icond4} of $y_{11}$ and property \ref{itm:y2icond4} of $y_{22}$\\
	If $a < x_1$ then $a < g$ by \eqref{eq:gminusxifirstcase} and $a \leq h_1 \vee h_2=g^-$. Let $b$ be a join-irreducible component of $a$ in $L_0 \langle y_1,y_2 \rangle$. We claim that $b \leq h_1$. We have that $b$ is join-irreducible in $L_0 \langle y_1,y_2 \rangle$ and $b \neq y_1,y_2$ because  $b < x_ 1$ and $y_1,y_2 \nless x_1$. Indeed, by equations \eqref{eq:yihminusyjk}, we have:
	\begin{equation}
	\begin{split}
	y_2-x_1=(y_2-y_{21})-y_{11}=y_{22}-y_{11}=y_{22} \neq 0 \\
	y_1-x_1=(y_1-y_{11})-y_{21}=y_{12}-y_{21}=y_{12} \neq 0 \\
	\end{split}
	\end{equation}
	Suppose $b \nleq h_1$, then by property \ref{itm:y1icond4} of $y_{11}$ we would get $b \nless y_{11}$, furthermore $b \nleq h_1 \wedge y_2$ and by property \ref{itm:y2icond4} of $y_{21}$ we would get $b \nless y_{21}$. Then $b$ would also be join-irreducible in $L_0 \langle y_{ij} \: | \: i,j=1,2 \rangle$ (see Claim 4). Therefore $b \nless y_{11} \vee y_{21}=x_1$ but this is absurd. Thus for any $b$ join-irreducible component of $a$ we have $b \leq h_1$ and hence $a \leq h_1$.\\ 
	For $x_2$ the reasoning is analogous
\end{enumerate}
\framebox{Subcase 2.2: $h_1,h_2$ comparables.} The remaining cases are when $h_1 < h_2$ or $h_2 < h_1$ since $h_1=h_2$ only occurs when $n=0$.\\
We now consider the case $h_1 < h_2$, for $h_2 < h_1$ the reasoning is analogous.\\
In this case, by equations \eqref{triplespredecessor}, we have $y_2^-=(h_1 \wedge y_2) \vee h_2=h_2$ in $L_0 \langle y_1,y_2 \rangle$. Notice that we can apply the inductive hypothesis only on the first signature in \eqref{triplesinductsplit} because $h_2 \nleq h_1$ but it is not true that $h_1 \nleq h_2$. Then we obtain the existence of $y_{11},y_{12}$ with the same properties \ref{itm:y1icond1}, \ref{itm:y1icond2}, \ref{itm:y1icond3}, \ref{itm:y1icond4} as in the previous subcase. We define $x_1 = y_{11}$ and $x_2=y_{12} \vee y_2$.\\
We want to prove that $x_1,x_2$ form a primitive couple $(x_1,x_2)$ of the second kind over $L_0$ inducing the signature $(h_1,h_2,g)$.\\
We have
\begin{equation} \label{eq:gminusxisecondcase}
g-x_1=x_2 \quad \text{ and } \quad g-x_2=x_1
\end{equation}
because
\begin{equation*}
\begin{split}
g-x_1=(y_1 \vee y_2)-y_{11}=(y_1-y_{11}) \vee (y_2-y_{11})=y_{12} \vee y_2=x_2 \\
g-x_2=(y_1 \vee y_2)-(y_{12} \vee y_2)=((y_1-y_2)-y_{12}) \vee ((y_2-y_2)-y_{12})=\\
y_1-y_{12}=y_{11}=x_1
\end{split}
\end{equation*}
We have used that $y_2 -y_{11}=y_2$, it is proven in the same way as \eqref{eq:yiminusyjkfirstcase} above.
\begin{enumerate}
\item 
	Equations \eqref{eq:gminusxisecondcase} imply $g=x_1 \vee x_2$, thus if $x_1=x_2$ then $x_1=g=0$ but this is absurd because $x_1,x_2 \neq 0$ since $y_{11}, y_{12}, y_2 \neq 0$ because they are not in $L_0$.\\
Furthermore $x_1,x_2 \notin L_0$ since $g$ is join-irreducible in $L_0$ and $g-x_1=x_2$ and $g-x_2=x_1$ are different from $0$ and $g$.
\item
	See equations \eqref{eq:gminusxisecondcase}.
\end{enumerate}
Let now $a$ be a join-irreducible element of $L_0$:
\begin{enumerate}[resume]
\item 
	If $a < g$ then $a \leq g^-= h_1 \vee h_2=h_2$.
	\begin{itemize}
	\item 
		If $a \leq h_1$ then $a-x_1=0$ because $h_1 \leq y_{11}=x_1$, moreover $a \leq h_1 < h_2 \leq y_2 < x_2$ imply $a-x_2=0$.
	\item 
		If $a \leq h_2$ and $a \nleq h_1$ clearly $a-x_2=0$ because $a \leq h_2 \leq y_2 \leq x_2$. We want to prove that $a-x_1=a$.
		\begin{claim}
		The join-irreducible components of $a$ in $L_0 \langle y_1,y_2 \rangle$ coincide with the join-irreducible components of $a$ in $L_0 \langle y_1,y_2, y_{11}, y_{12} \rangle$.
		\end{claim}
		\begin{claimproof}
		Since $a$ is the join of its join-irreducible components in $L_0 \langle y_1,y_2 \rangle$, it is sufficient to prove that any join-irreducible component $b$ of $a$ in $L_0 \langle y_1,y_2 \rangle$ is join-irreducible in $L_0 \langle y_1,y_2, y_{11}, y_{12} \rangle$. Pick such $a,b$. We have $b \neq y_1$ because $b \leq h_2$ and $y_1 \nleq h_2$. Notice that $y_1 \nleq h_2$ since $y_1 \leq h_2$ would imply $0=y_1-h_2=(g-y_2)-h_2=(g-h_2)-y_2=g-y_2=y_1$ which is absurd. Then we have that $b$ is also join-irreducible in $L_0 \langle y_1,y_2, y_{11}, y_{12} \rangle$, this follows from the fact that $L_0 \langle y_1,y_2, y_{11}, y_{12} \rangle$ is a minimal extension of $L_0 \langle y_1,y_2 \rangle$ of the second kind, which implies that the join-irreducibles in $L_0 \langle y_1,y_2 \rangle$ different from $y_1$ are still join-irreducible in $L_0 \langle y_1,y_2, y_{11}, y_{12} \rangle$.
		\end{claimproof}
		Since $a$ is join-irreducible of $L_0$ and $a \nleq h_1$ it is $a-h_1=a$. For any $b$ join-irreducible component of $a$ in $L_0 \langle y_1,y_2 \rangle$ we have $b \nleq h_1$ because $a-h_1=a$ means that $h_1$ is not greater than or equal to any join-irreducible component of $a$. Since $b \nleq h_1$ by property \ref{itm:y1icond4} of $y_{11}$ we have that $b \nless y_{11}=x_1$, therefore $b \nleq x_1$ because $b \neq y_{11}$ since $y_{11} \notin L_0 \langle y_1,y_2 \rangle$ . This implies that $a-x_1=a$ because $x_1$ is not greater than or equal to any join-irreducible component of $a$ in $L_0 \langle y_1,y_2, y_{11}, y_{12} \rangle$.
	\end{itemize}
\item 
	By property \ref{itm:y1icond4} of $y_{11}$ we have that $a \leq h_1$ iff $a < y_{11}=x_1$.\\
	If $a \leq h_2$ then $a \leq h_2 < y_2 \leq x_2$ since $h_2=y_2^-$.\\
	If $a < x_2$, since $ x_2 = y_{12} \vee y_2 \leq y_1 \vee y_2=g$ then $a < g$ and $a \leq g^-=h_1 \vee h_2=h_2$.
\end{enumerate}
\end{proof}

\subsection{Density axioms}

\noindent\textbf{[Density 1 Axiom]} For every $c$ there exists $b \neq 0$ such that $c \ll b$

\begin{theorem}
Any existentially closed CBS satisfies the Density 1 Axiom.
\end{theorem}

\begin{proof}
It is sufficient to show, by Lemma \ref{lemdjcobrouw}, that for any finite CBS $L_0$ and $c \in L_0$ there exists a finite extension $L_0 \subseteq L$ with $b \in L$ different from $0$ such that $c \ll b$.\\
Let $P_0$ be the finite poset dual to $L_0$ and $C$ its downset corresponding to $c$.\\
Let $P$ be the poset obtained by $P_0$ by adding a new maximum element $m \in P$ such that $m \geq p$ for any $p \in P_0$ and $\varphi:P \to P_0$ a surjective $\mathbf{P}$-morphism such that $\text{dom } \varphi=P_0$ and it is the identity on its domain. Then $C \ll \fregiu m$ and take as $L$ the CBS dual to $P$ and $b \in L$ corresponding to $\fregiu m$.
\end{proof}

\noindent\textbf{[Density 2 Axiom]} For every $c,a_1,a_2,d$ such that $a_1,a_2 \neq 0$, $c \ll a_1$, $c \ll a_2$ and $a_1-d=a_1$, $a_2-d=a_2$ there exists an element $b$ different from $0$ such that:
\begin{equation*}
\begin{split}
     c \ll b\\ 
     b \ll a_1 \\
     b \ll a_2 \\
     b-d=b
\end{split}
\end{equation*}

\begin{theorem}
Any existentially closed CBS satisfies the Density 2 Axiom.
\end{theorem}

\begin{proof}
It is sufficient to show, by Lemma \ref{lemdjcobrouw}, that for any finite CBS $L_0$ and $c,a_1,a_2,d$ such that $a_1,a_2 \neq 0$, $c \ll a_1$, $c \ll a_2$ and $a_1-d=a_1$, $a_2-d=a_2$ there exists a finite extension $L_0 \subseteq L$ with $b \in L$ different from $0$ such that $c \ll b$, $b \ll a_1$, $b \ll a_2$ and $b-d=b$.\\
Let $P_0$ the poset dual to $L_0$ and $C,A_1,A_2,D$ its downsets corresponding to $c,a_1,a_2,d$.\\
\framebox{If $C= \emptyset$} choose two maximal elements $\alpha^1, \alpha^2$ respectively of $A_1$ and $A_2$ and obtain a poset $P$ by adding a new element $\beta$ to $P_0$ and setting for any $x \in P$:
\begin{itemize}
\item $\beta \leq x$ iff $x=\beta$ or $\alpha^1 \leq x$ or $\alpha^2 \leq x$.\\
	If $\alpha^1, \alpha^2$ are incomparable they become the only two successors of $\beta$ in $P$, otherwise if e.g. $\alpha^1 \leq \alpha^2$ then $\alpha^1$ is the only successor of $\beta$.
\item $x \leq \beta$ iff $x=\beta$, i.e. $\beta$ is minimal in $P$.
\end{itemize}
Define a surjective $\mathbf{P}$-morphism $\varphi:P \to P_0$ taking $\text{dom } \varphi=P_0$ and $\varphi$ acting as the identity on its domain.
Take $B= \fregiu \beta$, we have:
\begin{itemize}
\item $\fregiu \varphi^{-1}(C)=\emptyset \ll B$,
\item $B \ll \fregiu \varphi^{-1}(A_1)=A_1 \cup \lbrace \beta \rbrace$,
\item $B \ll \fregiu \varphi^{-1}(A_2)=A_2 \cup \lbrace \beta \rbrace$,
\item $B- \fregiu \varphi^{-1}(D)=B$.\\
	Indeed, since $a_1-d=a_1$ and $a_2-d=a_2$, $D$ does not contain any maximal element of $A_1$ or $A_2$, in particular it does not contain $\alpha^1$ or $\alpha^2$
\end{itemize}
Take $L$ the CBS dual to $P$ and $b \in L$ corresponding to $B$.\\
\framebox{If $C \neq \emptyset$} let $\gamma_1, \ldots, \gamma_n$ be the maximal elements of $C$.\\
Choose for any $i=1, \ldots,n$ two maximal elements $\alpha_i^1, \alpha_i^2$ respectively of $A_1$ and $A_2$ such that $\gamma_i \leq \alpha_i^1$ and $\gamma_i \leq \alpha_i^2$. Notice that they exist and $\gamma_i \neq \alpha_i^1$, $\gamma_i \neq \alpha_i^2$  because $C \ll A_1$ and $C \ll A_2$.\\
Obtain a poset $P$ by adding new elements $\beta_1, \ldots, \beta_n$ to $P_0$ and setting for any $x \in P$:
\begin{itemize}
\item $\beta_i \leq x$ iff $x=\beta_i$ or $\alpha_i^1 \leq x$ or $\alpha_i^2 \leq x$.\\
	If $\alpha_i^1, \alpha_i^2$ are incomparable they become the only two successors of $\beta_i$ in $P$, otherwise if e.g. $\alpha_i^1 \leq \alpha_i^2$ then $\alpha_i^1$ is the only successor of $\beta_i$.
\item $x \leq \beta_i$ iff $x=\beta$ or $x \leq \gamma_i$,\\
	i.e. $\gamma_i$ is the unique predecessor of $\beta_i$ in $P$.
\end{itemize}
Define a surjective $\mathbf{P}$-morphism $\varphi:P \to P_0$ taking $\text{dom } \varphi=P_0$ and $\varphi$ acting as the identity on its domain.\\
Take $B= \fregiu \beta_1 \cup \cdots \cup \fregiu \beta_n $, we have:
\begin{itemize}
\item $\fregiu \varphi^{-1}(C) \ll B$,
\item $B \ll \fregiu \varphi^{-1}(A_1)=A_1 \cup \lbrace \beta_1, \ldots, \beta_n \rbrace$,
\item $B \ll \fregiu \varphi^{-1}(A_2)=A_2 \cup \lbrace \beta_1, \ldots, \beta_n \rbrace$,
\item $B- \fregiu \varphi^{-1}(D)=B$.\\
	Indeed $D$ does not contain any maximal element of $A_1$ or $A_2$, in particular it does not contain $\alpha_i^1$ or $\alpha_i^2$ for any $i=1, \ldots,n$.
\end{itemize}
Take $L$ the CBS dual to $P$ and $b \in L$ corresponding to $B$.
\end{proof}

\begin{theorem}
Let $L$ be a CBS satisfying the Splitting, Density 1 and Density 2 Axioms.\\
Then for any finite sub-CBS $L_0 \subseteq L$ and for any signature $(h,G)$ of the first kind in $L_0$ there exists a primitive element $x \in L$ of the first kind over $L_0$ inducing such signature.
\end{theorem}

\begin{proof}
We follow this strategy: if $G= \emptyset$ we use the Density 1 Axiom to take an element $m \in L$ greater than any element of $L_0$, then, thanks to the Splitting Axiom, using Theorem \ref{theosecondextaxiom} we `split' $m$ into two elements, one over $1_{L_0}$ the top element of $L_0$ and another over $h$. It turns out that this second element is primitive of the first kind and induces the signature $(h, \emptyset)$.
If $G$ is nonempty and $G=\lbrace g_1, \ldots, g_k \rbrace$ we suppose to have already found a primitive element $y$ inducing the signature $(h,\lbrace g_1, \ldots, g_{k-1} \rbrace)$. Then, using Theorem \ref{theosecondextaxiom} again, we `split' $g_k$ into two elements $g_k',g_k''$, the first over $h$ and the second over the predecessor of $g_k$. Finally, applying the Density 2 Axiom, we obtain an element of $L$ in between $h$, $g_k'$ and $y$. It turns out that this element is primitive of the first kind and induces the signature $(h, G)$.\\
The statement of the Theorem require that, according to the definition of primitive element inducing a given signature, we need to do the following. Given $h \in L_0$ and $G$ a set of join-irreducibles of $L_0$ such that $h < g$ for any $g \in G$, we have to find $x \in L$ such that:
\begin{enumerate}
\item $x \notin L_0$ \label{itm:firstextdens2axcond1}
\end{enumerate}
and for any $a$ join-irreducible of $L_0$:
\begin{enumerate}[resume]
\item $a-x \in L_0$, \label{itm:firstextdens2axcond2}
\item either $ x-a= x \text{ or } x-a= 0$, \label{itm:firstextdens2axcond3}
\item $a < x$ iff $a \leq h \;$ and $\; x < a$ iff $g_i \leq a$ for some $i=1, \ldots, k$. \label{itm:firstextdens2axcond4}
\end{enumerate}
The proof is by induction on $k= \# G$.\\
\framebox{Case $k=0$.}\\
Let $1_{L_0}$ be the maximum element of $L_0$, by Density 1 there exists $0 \neq m \in L$ such that $1_{L_0} \ll m$.\\
Then $L_1=L_0 \cup \lbrace m \rbrace$ is a sub-CBS of $L$. Indeed it is closed under taking joins and differences since for any $a \in L_0$ we have $m > a$ and thus $a-m=0$ and $m = m-1_{L_0} \leq m-a \leq m$, therefore $m-a=m$. Hence $m$ is a join-irreducible of $L_1$. Furthermore it is clear that the join-irreducibles of $L_1$ are the join-irreducibles of $L_0$ and $m$.\\
$(h,1_{L_0},m)$ is a signature of the second kind in $L_1$, indeed $h \vee 1_{L_0}=1_{L_0}=m^-$.\\
Thanks to the Splitting Axiom we can apply Theorem \ref{theosecondextaxiom} to the signature $(h,1_{L_0},m)$ in $L_1$ and obtain the existence of a primitive couple of the second kind $(x_1,x_2) \in L^2$ inducing such signature. 
Thus we have that there exist $x_1,x_2 \in L$ such that:
\begin{enumerate}
\item
	$x_1 \neq x_2$ and $x_1,x_2 \notin L_1$, \label{itm:density1xiproperty1}
\item
	$m-x_1=x_2$ and $m-x_2=x_1$ \label{itm:density1xiproperty2}
\end{enumerate}
and for any $c$ join-irreducible of $L_1$:
\begin{enumerate}[resume]
\item
	if $c < m$ then $c-x_i \in L_1$ for $i=1,2$, \label{itm:density1xiproperty3}
\item
	$c < x_1$ iff $c \leq h \;$ and $c < x_2$ iff $c \leq 1_{L_0}$. \label{itm:density1xiproperty4}
\end{enumerate}
Recall that Lemma \ref{defprim2} implies that for any $c \in L_1$:
\begin{enumerate}[label=(\roman*)] 
\item
	$c-x_i \in L_1$ or $c-x_i= b \vee x_j $ for some $b \in L_1$ with $\lbrace i,j \rbrace = \lbrace 1,2 \rbrace$. \label{itm:density1xipropertyi}
\item
	$x_i-c= x_i \text{ or } x_i-c= 0 \;$ for $i=1,2$. \label{itm:density1xipropertyii}
\end{enumerate}
Let $x=x_1$, it is the element we were looking for. Indeed we now show that $x$ is a primitive element of the first kind over $L_0$ inducing the signature $(h, \emptyset)$
\begin{enumerate}
\item 
	 $x \notin L_0$ since $x=x_1 \notin L_1$ by property \ref{itm:density1xiproperty1} of $x_1$.
\end{enumerate}
Let $a$ be a join-irreducible of $L_0$. Then
\begin{enumerate}[resume]
\item
	$a-x_1 \in L_0$.
	Indeed, from $a \leq 1_{L_0}$ it follows (by property \ref{itm:density1xiproperty4} of $x_2$) $a < x_2$; then by \ref{itm:density1xipropertyii} either $a-x_1 \in L_1$ or $a-x_1=b \vee x_2$ with $b \in L_0$. The latter is absurd because (for $a < x_2$) we would get $x_2 > a \geq a-x_1=b \vee x_2 \geq x_2$. Then $a-x_1 \in L_1$, i.e. $a-x_1 \in L_0$ because $m > a \geq a-x_1$.
\item
	$x_1-a=x_1$ or $x_1-a=0$ by property \ref{itm:density1xipropertyii}.
\item
	$a < x_1$ if and only if $a \leq h$ by property \ref{itm:density1xiproperty4} of $x_1$.\\
	$x_1 \nless a$, because if $x_1 < a$ then $x_1 < 1_{L_0}$ and so $0=x_1-1_{L_0}=(m-x_2)-1_{L_0}=(m-1_{L_0})-x_2=m-x_2=x_1$ which is absurd because $x_1 \notin L_1$ by property \ref{itm:density1xiproperty1} of $x_1$.
\end{enumerate}
\framebox{Case $k \geq 1$.}\\
Suppose that $G= \lbrace g_1, \ldots, g_k \rbrace$. By inductive hypothesis there exists a primitive element $y \in L$ of the first kind over $L_0$ which induces the signature $(h, \lbrace g_1, \ldots, g_{k-1} \rbrace)$. This means that for any $a$ join-irreducible of $L_0$:
\begin{enumerate}
\item $y \notin L_0$, \label{itm:density2yproperty1}
\item $a-y \in L_0$, \label{itm:density2yproperty2}
\item either $ y-a= y \text{ or } y-a= 0$, \label{itm:density2yproperty3}
\item $a < y$ iff $a \leq h \;$ and $\; y < a$ iff $g_i \leq a$ for some $i=1, \ldots, k-1$. \label{itm:density2yproperty4}
\end{enumerate}
Recall that Lemma \ref{lemmaprim1} shows that the properties \ref{itm:density2yproperty2} and \ref{itm:density2yproperty3} actually hold for any $a \in L_0$.\\
Notice that $g_k$ is still join-irreducible in the sub-CBS $L_0 \langle y \rangle \subseteq L$ generated by $L_0$ and $y$ since $L_0 \subseteq L_0 \langle y \rangle$ is a minimal finite extension of the first kind by Theorem \ref{theoext1}.\\
Since $L$ satisfies the Splitting Axiom, we can apply Theorem \ref{theosecondextaxiom} to the signature $(h,g_k^-, g_k)$ in $L_0 \langle y \rangle$. Notice that it is a signature of the second kind because $h \vee g_k^-= g_k^- \ll g_k$. Therefore, there exists a primitive couple of the second kind $(g_k',g_k'') \in L^2$ inducing such signature. Thus we have that there exist $g_k',g_k'' \in L$ such that:
\begin{enumerate}
\item $g_k',g_k'' \notin L_0 \langle y \rangle$ and $g_k' \neq g_k''$, \label{itm:density2gkproperty1}
\item $g_k-g_k'=g_k''$ and $ g_k-g_k''=g_k'$ \label{itm:density2gkproperty2}
\end{enumerate}
and for any $a$ join-irreducible of $L_0 \langle y \rangle$:
\begin{enumerate}[resume]
\item if $a<g_k$ then $a-g_k' \in L_0 \langle y \rangle$ and $a-g_k'' \in L_0 \langle y \rangle$, \label{itm:density2gkproperty3}
\item $a < g_k'$ iff $a \leq h \quad$ and $\quad a < g_k''$ iff $a \leq g_k^- $. \label{itm:density2gkproperty4}
\end{enumerate}
Observe that property \ref{itm:density2gkproperty4} actually holds for any $a \in L_0 \langle y \rangle$ since any element in a finite CBS is the join of join-irreducible elements.\\
Apply the Density 2 Axiom on $h,y,g_k',d$ where 
\[
d= \bigvee \lbrace b \text{ join-irreducible of $L_0$ s.t. } b \ngeq g_1, \ldots, b \ngeq g_k \rbrace.
\]
We can apply it because:\\
$h \ll y$ since by property \ref{itm:density2yproperty3} of $y$ we have $y-h=y$ because $h \in L_0$ and $h < y$.\\
$h \ll g_k'$ since $h < g_k'$ and $g_k'-h=(g_k-g_k'')-h=(g_k-h)-g_k''=g_k-g_k''=g_k'$. Notice that $g_k-h=g_k$ because $g_k$ is join-irreducible in $L_0$. \\
$y-d=y$ since for any $b$ join-irreducible in $L_0$ such that $b \ngeq g_1, \ldots, b \ngeq g_k$ we have $y-b=y$: otherwise, since $y$ is join-irreducible in $L_0 \langle y \rangle$, it would be $y-b=0$ so $b > y$ and then by property \ref{itm:density2yproperty4} of $y$ we would have $b \geq g_i$ for some $i<k$ which is absurd.\\
$g_k'-d=g_k'$ since 
\[
g_k = g_k-\bigvee \lbrace b \text{ join-irreducible of $L_0$ s.t. } b \ngeq g_k \rbrace \leq g_k-d \leq g_k
\]
and $g_k'-d=(g_k-g_k'')-d=(g_k-d)-g_k''=g_k-g_k''=g_k'$.\\
Then by the Density 2 Axiom there exists $0 \neq x \in L$ such that 
\begin{equation} \label{eq:density2casek>1}
h \ll x, \: x \ll y, \: x \ll g_k' \text{ and } x-d=x
\end{equation}
$x$ is the element we were looking for. Indeed, it is primitive of the first kind over $L_0$ and induces the signature $(h,G)$:
\begin{enumerate}
\item
	We have $x \notin L_0$ because if $x \in L_0$ then since $x < y$ it would be $x \leq h$ by property \ref{itm:density2yproperty4} of $y$ but this is absurd because $x \neq 0$ and $h \ll x$.
\end{enumerate}
Let $a$ be a join-irreducible of $L_0$:
\begin{enumerate}[resume]
\item 
	If $a \leq h$ then $a-x=0$ since $h \leq x$ by \eqref{eq:density2casek>1}.\\
	If $a \nleq h$ then by property \ref{itm:density2gkproperty4} of $g_k'$ we have $a \nless g_k'$.
	\begin{itemize}
	\item
		If $a \nleq h$ and $a \neq g_k$ then $a$ is still join-irreducible in $L_0 \langle y,g_k',g_k'' \rangle$ (since $L_0 \langle y \rangle \subseteq L_0 \langle y,g_k',g_k'' \rangle$ is a minimal finite extension by Theorem \ref{theoext2}), thus $a-g_k'=a$. Therefore $a-x=a$ because $a=a-g_k' \leq a-x \leq a$ since $x \leq g_k'$.
	\item
		If $a=g_k$ then by $x \ll g_k'$ (see \eqref{eq:density2casek>1})
	\begin{equation*}
	\begin{split}
     g_k-x=(g_k' \vee g_k'')-x=(g_k'-x) \vee (g_k''-x )=g_k' \vee ((g_k-g_k')-x)\\
     =g_k' \vee (g_k-(g_k' \vee x)) =g_k' \vee (g_k-g_k')=g_k' \vee g_k''=g_k.
	\end{split}
	\end{equation*}
	\end{itemize}
\item 
	If $a \geq g_i$ for some $i=1, \ldots,k$ then:
	\begin{itemize}
	\item If $i \neq k$ then $a \geq y \geq x$ and $x-a=0$ by property \ref{itm:density2yproperty4} of $y$ and \eqref{eq:density2casek>1}.
	\item If $i=k$ then $a \geq g_k \geq g_k' \geq x$ and $x-a=0$.
	\end{itemize}
	If $a \ngeq g_i$ for any $i=1,\ldots,k$ then $a \leq d$ and $x-a=x$ since $x=x-d \leq x-a \leq x$
\item If $a < x$ then $a < g_k'$ and thus $a \leq h$ by property \ref{itm:density2gkproperty4} of $g_k'$.\\
	If $a \leq h$ then $a < x$ because $h <x$ by \eqref{eq:density2casek>1}.\\
	If $x <a$ and $a \ngeq g_1, \ldots, a \ngeq g_k$, then $a \leq d$ and $x=x-d \leq x-a =0$ which is absurd, thus $g_i \leq a$ for some $i=1, \ldots,k$.\\
	If $g_i \leq a$ for some $i=1, \ldots,k$ then:
	\begin{itemize}
	\item If $i \neq k$ then, since $y < g_i$ by property \ref{itm:density2yproperty4} of $y$, we have $x <y < g_i \leq a$ and thus $x < a$.
	\item If $i=k$ then $x < g_k' <g_k \leq a$ and thus $x<a$.
	\end{itemize}
\end{enumerate}
\end{proof}

\section{Properties of existentially closed CBSes} \label{section:properties}

From our investigation we can easily obtain some properties of the existentially closed CBSes:

\begin{proposition}
If $L$ is an existentially closed CBS, then $L$ does not have a maximum element.
\end{proposition}

\begin{proof}
Since $L$ satisfy the Density 1 Axiom for any $c \in L$ there exists an element $b \neq 0$ such that $c \ll b$ and therefore $c <b$. This implies that there cannot be a maximum element of $L$.
\end{proof}

\begin{proposition}
Let $L$ be an existentially closed CBS and $a,b \in L$.\\
If $a$ and $b$ are incomparable, i.e. $a \nleq b $ and $b \nleq a$, then there does not exist the meet of $a$ and $b$ in $L$.\\
Notice that if $a \leq b$ then the meet exists and it is $a$.
\end{proposition}

\begin{proof}
Denote by $c$ the meet of $a$ and $b$. Consider $L_0 \subseteq L$ the sub-CBS generated by $a,b,c$. It is finite by local finiteness. $c$ is the meet of $a$ and $b$ also in $L_0$.\\
Since $a,b$ are incomparable there exist $g_1,g_2$ join-irreducible components in $L_0$ respectively of $a$ and $b$ such that $g_1 \nleq b$ and $g_2 \nleq a$.\\
By Theorem \ref{theoaxiominterm} taking $h=0 \in L_0$ we have that there exists $x \in L \setminus L_0$ such that for any $d \in L_0$:
\begin{itemize}
\item $d < x$ iff $d=0$,
\item $ x < d$ iff $g_i \leq d$ for $i=1$ or $i=2$.
\end{itemize}
We have that $x \nleq c$ since  $x \notin L_0$, $g_1 \nleq c$ and $g_2 \nleq c$, therefore $c <  c \vee x$. Notice that $x < g_1 \leq a$ and $x < g_2 \leq b$, thus $c \vee x \leq a$ and $c \vee x \leq b$. This implies that $c$ cannot be the meet of $a$ and $b$ in $L$.
\end{proof}

\begin{proposition}
If $L$ is an existentially closed CBS, then there are no join-irreducible elements of $L$.
\end{proposition}

\begin{proof}
Let $g$ be a nonzero element of $L$. We can apply the splitting axiom on the triple $g,0,0$, then there exist $g_1,g_2 \in L$ such that 
\[
g-g_1=g_2, \quad g-g_2=g_1 \quad \text{and} \quad g_1,g_2 \neq 0.
\]
Since $g_1,g_2 \leq g$ and $g-(g_1 \vee g_2)=(g-g_1)-g_2=0$ we have that $g=g_1 
\vee g_2$. Moreover $g_1,g_2 \neq g$ because $g_1,g_2 \neq 0$. Therefore $g$ cannot be join-irreducible because $g=g_1 \vee g_2$ with $g_1,g_2 \neq g$, recall that $0$ is never join-irreducible.
\end{proof}

\bibliographystyle{alpha}
\bibliography{bibliography}

\end{document}